\documentclass[12pt,reqno]{amsart}

\usepackage[LGR,T1]{fontenc}
\usepackage[utf8]{inputenc}

\title{An entropic interpolation problem for incompressible viscid fluids}
\author[Arnaudon]{Marc Arnaudon}
\author[Cruzeiro]{Ana Bela Cruzeiro}
\author[Léonard]{Christian Léonard}
\author[Zambrini]{Jean-Claude Zambrini}

\date{April 4, 2017}

\usepackage{amssymb, amsmath, amsfonts, latexsym, enumerate}
\usepackage[usenames,dvipsnames]{pstricks}
\usepackage{epsfig}
\RequirePackage[colorlinks,linkcolor=blue,citecolor=blue,urlcolor=blue]{hyperref}

\usepackage[english]{babel}

\newtheorem{theorem}{Theorem}
\newtheorem{lemma}[theorem]{Lemma}
\newtheorem{proposition}[theorem]{Proposition}
\newtheorem{corollary}[theorem]{Corollary}

\newtheorem{definition}[theorem]{Definition}
\newtheorem{definitions}[theorem]{Definitions}

\newtheorem{assumption}[theorem]{Assumption}
\newtheorem{hypotheses}[theorem]{Hypotheses}

\theoremstyle{remark}
\newtheorem{remark}[theorem]{Remark}
\newtheorem{remarks}[theorem]{Remarks}

\numberwithin{theorem}{section}

\newcommand{\RR}{\mathbb{R}}
\newcommand{\Rn}{\mathbb{R}^n}

\newcommand{\1}{\mathbf{1}}

\newcommand{\ttimes}{\!\times\!}

\newcommand\pf{_{\#}}
\newcommand{\vol}{\mathrm{vol}}

\newcommand{\as}{\textrm{-a.s.}}
\renewcommand{\ae}{\textrm{-a.e.}}
\newcommand{\Id}{\mathrm{Id}}
\newcommand{\scal}{\!\cdot\!}

    \DeclareMathOperator{\supp}{supp}

\newcommand{\lsc}{lower semicontinuous}

\newcommand{\ud}{\frac{1}{2}}

\newcommand\XX{\mathcal{X}}
\newcommand\XXX{\XX^2}
\newcommand\PX{\mathrm{P}(\XX)}

\newcommand\PXX{\mathrm{P}(\XXX)}
\newcommand\MX{\mathrm{M}(\XX)}

\newcommand\PO{\mathrm{P}(\Omega)}
\newcommand\MO{\mathrm{M}(\Omega)}

\newcommand\OO{\Omega}

\newcommand\BO{\mathrm{B}(\Omega)}

\newcommand\ii{{[0,1]}}
\newcommand\iX{{[0,1]\ttimes\XX}}
\newcommand\TX{{ \mathcal{T}\ttimes\XX}}
\newcommand\IX{\int_{\XX}}
\newcommand\IXX{\int_{\XXX}}

\newcommand\Iii{\int_\ii}
\newcommand\IT{\int _{ \mathcal{T}}}
\newcommand\IiX{\int_\iX}
\newcommand\ITX{\int_\TX}

\newcommand{\IRn}{\int _{\Rn}}

\newcommand{\BXX}{B(\XXX)}
\newcommand{\BTX}{B(\TX)}

\newcommand{\ma}{ \mu ^{ \alpha}}

\newcommand{\vf}{\overset{\rightharpoonup}{v}}
\newcommand{\vb}{\overset{\leftharpoonup}{v}}

\newcommand\vvf[2]{\overset{#1\,\rightarrowtail}{v _{ #2}}}
\newcommand\vvb[2]{\overset{\leftarrowtail\, #1}{v _{ #2}}}
\newcommand{\Py}{\overset{\leftarrow\, y}P}

\newcommand{\grad}{\nabla}

\newcommand{\Tn}{ \mathbb{T}^n}

\newcommand{\TT}{ \mathcal{T}}
\renewcommand{\SS}{ \mathcal{S}}
\newcommand{\PT}{\mathrm{P}(\TT)}


 \address{Institut de mathématiques de Bordeaux, Université de Bordeaux, 33405 Talence Cedex, France}
 \email{marc.arnaudon@math.u-bordeaux.fr}
 
\address{GFMUL and Dep.\,de Mat. Instituto Superior Técnico, Av.\ Rovisco Pais, 1049-001 Lisboa, Portugal}
\email{abcruz@math.tecnico.ulisboa.pt}

\address{Modal'X, UPL, Univ Paris Nanterre, F92000 Nanterre France}
\email{christian.leonard@u-paris10.fr}

\address{GFMUL and Dep.\,de Mat. Faculty of Sciences.
Campo Grande, Edifício C6. PT-1749-016 Lisboa. Portugal}
\email{jczambrini@fc.ul.pt }

\begin{document}

\maketitle

\begin{abstract} 
In view of studying incompressible inviscid fluids,  Brenier introduced in the late 80's a relaxation of a geodesic problem addressed by Arnold in 1966. Instead of \emph{inviscid} fluids, the present paper is devoted to incompressible \emph{viscid} fluids. A natural analogue of Brenier's  problem is introduced,  where  generalized flows are no more supported  by absolutely continuous paths, but by Brownian sample paths. It turns out that this new variational problem is an entropy minimization problem with marginal constraints entering the class of convex minimization problems.

This paper explores the connection between this variational problem and Brenier's  original problem. Its dual problem is derived and the general shape of its solution  is described. Under the restrictive assumption that the pressure is a nice function, the kinematics of its solution is made explicit and its connection with the Navier-Stokes equation is established.

\end{abstract}

\tableofcontents

\section*{Introduction}

In the article \cite{Arn66}, Arnold addressed a geodesic problem on the manifold of  all volume preserving diffeomorphisms on the torus. The resulting geodesics offer us a natural description of the evolution of an incompressible perfect fluid.  Unfortunately, very little is known on the global existence of these geodesics \cite{EM70} and Shnirelman proved in \cite{Shn85,Shn94} that solutions do not exist in general. In the seminal article \cite{Bre89}, Brenier introduced a relaxation of Arnold's problem in terms of generalized flows. In this extended setting, global existence of generalized ``minimizing flows'' is much easier to obtain. 

While both Arnold's and Brenier's problems are related to the evolution of  \emph{inviscid} fluids, usually described by the Euler equation \eqref{eq-05}, the present article introduces a stochastic analogue of Brenier's relaxed problem related to the description of the evolution of \emph{viscid} fluids. This viscosity is usually modeled by adding an extra Laplacian term to the Euler equation, leading to the Navier-Stokes equation \eqref{eqac-04}. Following Itô's   stochastic description of parabolic diffusion equations, we shall introduce stochastic differential equations based on Brownian motion to take the viscosity into account. 

Brenier's problem \eqref{eq-08} amounts to minimize an average kinetic action, the averaging procedure being performed over the set of all absolutely continuous sample paths. This is no longer available when the sample paths are  nowhere differentiable Brownian trajectories. However, it is still possible to consider some kinetic action in terms of Nelson's stochastic velocities. It happens that this stochastic action is a relative entropy (with respect to the Wiener measure) and that Brenier's minimization of an average kinetic action turns out to be an entropy minimization problem: the Bredinger  problem \eqref{eq-bdg} stated below at page \pageref{eq-bdg}.

\subsection*{Literature}

Brenier's relaxation of Arnold's geodesic problem was introduced in  \cite{Bre89b}. Its dual problem was established and investigated for the first time in  \cite{Bre93} where it was emphasized  that the pressure field is the natural Lagrange multiplier of the incompressibility condition. 
The connection between solutions to Brenier's problem and the notion of measure-valued solutions to the Euler equation in the sense of DiPerna and Majda was  clarified in  \cite{Bre99}. 
The regularity of the pressure field was explored in   \cite{Bre93} and revisited in   \cite{Bre99}. Further improvements about the dual problem and the regularity of the pressure were obtained later by Ambrosio and Figalli in \cite{AF08,AF09}.

Considering viscid fluids (Navier-Stokes equation) instead of inviscid ones, we refer to \cite{AAC14,ACF16} for  works  where generalized flows are also considered. The present article is also about generalized flows in the setting of viscid fluids but with an alternate point of view. To our knowledge, it is the first attempt to extend Brenier's \emph{variational approach}  in this context.

\subsection*{Outline of the paper}

More about the connection between Bredinger and Brenier problems and their relations with fluid dynamics is given at Section \ref{sec-bre}. Since Bredinger's problem enters the class of convex minimization problems, it admits a natural dual problem; this is exposed at Section \ref{sec-dual}. It is the place where the pressure enters the game. The general shape of the Radon-Nikodym density of the solution of Bredinger's problem with respect to the Wiener measure is described at Section \ref{sec-shape}. It is shown that, in general, this solution fails to be Markov but  remains reciprocal. At Section \ref{sec-regular}, under the restrictive assumption that the pressure is a nice function, the kinematics of the solution is made explicit. This permits us to establish a connection between Bredinger's problem and the Navier-Stokes equation. Finally, the last Section \ref{sec-existence} is devoted to a characterization of the existence of a solution to Bredinger problem when the state space is the torus $\Tn$. 

In the present article,  the difficult problem of the regularity of  the pressure field is left apart, see Remark \ref{rem-01}(b).

\subsection*{Notation}
 The sets  of all Borel probability measures  on a topological set $\mathcal{Z}$ is denoted by $\mathrm{P}(\mathcal{Z}).$ 
For any probability measure $m\in \mathrm{P}(A)$ on the Borel set $A$, the push-forward of $m$ by the Borel measurable mapping $\theta:A\to B$ is denoted by $\theta\pf m\in \mathrm{P}(B)$ and  defined by $\theta\pf m(db):=m( \theta ^{-1}(db)),$ for any Borel subset $db\subset B.$

\subsubsection*{State space}

In general the state space $\XX$ for the fluid will be either the flat torus $\XX= \mathbb{T}^n=\Rn/\ii ^n$ or the whole space $\XX=\Rn.$ We denote it by $\XX$ when its specific structure does not matter.
 
\subsubsection*{Path space}

We denote  $\OO:=C(\ii,\XX)$ the set of all continuous paths  from the unit time interval $\ii$ to  $\XX$.
As usual, the canonical process is defined by
\begin{equation*}
X_t(\omega)=\omega_t\in\XX,\quad t\in\ii, \ \omega=(\omega_s)_{0\le s\le1}\in\OO
\end{equation*}
and $\OO$ is equipped with the canonical $\sigma$-field $\sigma(X_t; 0\le t\le 1)$. 

\subsubsection*{Marginal measures}

For any $Q\in\PO$ and $0\le s, t\le 1,$ we denote $Q_t=(X_t)\pf Q\in \PX$ and $Q _{st}=(X_s,X_t)\pf Q\in\PXX;$ they are respectively the laws of the position $X_t$ at time $t$ and of the couple of positions $(X_s,X_t)$ when the law of the whole random trajectory is $Q.$ In particular, taking $s=0$ and $t=1,$ $Q _{01}:=(X_0,X_1)\pf Q\in\PXX$ is the endpoint projection of $Q$ onto $\XXX.$ If $Q$ describes the random behavior of a particle, then $Q_t$ and $Q _{01}$ describe respectively the random behaviors of the particle at time $t$ and  the couple of endpoint positions $(X_0,X_1).$ We denote $Q^x:=Q(\cdot\mid X_0=x)\in\PO$ and the bridge of $Q$ between $x$ and $y\in\XX$ is  $Q ^{xy}(\cdot):=Q(\cdot\mid X_0=x,X_1=y)\in\PO.$ In particular, as $\XX$ is  Polish, the following disintegration formulas
\begin{equation*}
Q(\cdot)=\IX Q^x(\cdot)\, Q_0(dx)=\IXX Q ^{xy}(\cdot)\, Q _{01}(dxdy)\in\PO
\end{equation*}
are meaningful, i.e.\ $x\in\XX\mapsto Q^x\in\PO$ and $(x,y)\in\XXX\mapsto Q ^{xy}\in\PO$ are measurable kernels.

\subsection*{Relative entropy}

Let $ \mathcal{Y}$ be a measurable space. We denote $ \mathrm{P}( \mathcal{Y})$ and $ \mathrm{M}( \mathcal{Y})$ respectively the sets of probability and positive measures on $ \mathcal{Y}.$
The relative entropy of $ \textsf{q}\in \mathrm{M}(\mathcal{Y})$ with respect to the reference  measure $ \textsf{r}\in \mathrm{M}(\mathcal{Y})$ is defined by
\begin{equation*}
H( \textsf{q}| \textsf{r}):=\int _{\mathcal{Y}}\log \left(\frac{d \textsf{q}}{d \textsf{r}}\right)\,d \textsf{q}\in (-\infty,\infty]
\end{equation*}
whenever the integral is meaningful, i.e.\ when $ \textsf{q}$ is absolutely continuous with respect to $ \textsf{r}$ and $\int _{\mathcal{Y}}\log_- \left(d \textsf{q}/d \textsf{r}\right)\,d \textsf{q}<\infty.$
\\
When $ \textsf{r}\in \mathrm{P}( \mathcal{Y})$ is a probability measure,
for all probability measures $\textsf{q}\in \mathrm{P}( \mathcal{Y}),$ we have
\begin{align}\label{eq-01}
H(\textsf{q}|\textsf{r})=\min H(\cdot|\textsf{r})=0 \iff \textsf{q}=\textsf{r}.
\end{align}
A frequent use will be made  of the additive  decomposition formula
\begin{align}\label{eq-02}
H(\textsf{q}|\textsf{r})=H( f\pf \textsf{q}| f\pf \textsf{r})+\int _{ \mathcal{Z}} H( \textsf{q} ^{ f=z}| \textsf{r} ^{ f=z})\, (f\pf \textsf{q})(dz)
\end{align}
where $f: \mathcal{Y}\to \mathcal{Z}$ is any  measurable mapping between the Polish spaces $ \mathcal{Y}$ and $ \mathcal{Z}$ equipped with their Borel $ \sigma$-fields and $\textsf{q} ^{ f=z}$ is a regular version of the conditioned probability measure $\textsf{q}(\cdot\mid f=z).$ 

\begin{itemize}
\item
Note that as a definition $ \textsf{q} ^{ f=z}$ is always  a probability measure, even when $ \textsf{q}$ is not.
\item
It is necessary that $f\pf \textsf{q}$ is $ \sigma$-finite for the conditional probability measure $ \textsf{q} ^{ f}$ to be defined properly.
\end{itemize}
In particular, we see with \eqref{eq-01} and \eqref{eq-02} that 
\begin{align}\label{eq-03}
H(f\pf \textsf{p}|f\pf \textsf{r})\le H( \textsf{q}| \textsf{r}),
\end{align}
expressing the well-known property of decrease of the entropy by measurable push-forward.
\\
Taking $f=X_0:\OO\to \XX$ in \eqref{eq-02} gives, for any $Q,R\in\MO,$
\begin{align*}
H(Q|R)=H(Q_0|R_0)
	+\IX H(Q^x|R^x)\, Q_0(dx)
\end{align*}
whenever $Q_0$ and $R_0$ are $ \sigma$-finite and $H(Q_0|R_0)$ makes sense in $( - \infty, \infty].$
An interesting situation where unbounded path measures arise naturally is when the initial marginals $$Q_0=R_0=\vol$$ are prescribed to be the volume measure on $\XX=\Rn.$ In this case,
\begin{align*}
H(Q|R)=\IX H(Q^x|R^x)\, \vol(dx).
\end{align*}

\section{Bredinger's problem}\label{sec-bre}

The main role of  this article is played by the Bredinger problem, an entropy minimization problem stated below at \eqref{eq-bdg}. The present section is devoted to a brief exposition of some relations between Bredinger's problem and the evolution of an incompressible \emph{viscid} fluid. As our approach follows Brenier's one, we start with Brenier's problem and its relation with the evolution of an incompressible \emph{inviscid} fluid.

\subsection*{Eulerian and Lagrangian coordinates}
These coordinates correspond to two different descriptions of the same fluid flow through space and time. Let the state space $\XX$ be a subset of $\Rn.$
\begin{enumerate}
\item
The \emph{Eulerian} specification of the flow field in   $\XX$ is a vector field
$$
 (t,x)\in \ii\times \XX\mapsto   {v}\left(t,x\right)\in \Rn
$$
giving the  velocity at position $x$ and time $t$.
One  looks at the fluid motion  focusing on specific locations in  $\XX$. 
 
\item
The \emph{Lagrangian} specification of the flow field is a function
$$
 (t,x)\in\ii\times \XX \mapsto   q(t,x)\in \XX
$$
giving the position at time $t$ of the parcel which was located at $x$ at time $t=0$. One looks at fluid motion  following an individual particle. 
The labeling of the fluid particles  allows  keeping track of the  changes of the shape of fluid parcels over time.  
\end{enumerate}
The two specifications are related by:
$
    {v}\left(t,q(t,x)\right) = \partial_t q(t,x).
$
The total rate of change of a function or a vector field $F(t,z)$ experienced by a specific flow parcel is $${\mathrm{D}_t{F}}(t,z) =  \frac{d}{dt}F(t,q(t,x)){} _{\big| x=q_t^{-1}(z)}$$ where $z$ is fixed.
This  gives
$	
 {\mathrm{D}_t{F}} = ({\partial_t } + {v}\scal \nabla) {F},
$	
since 
\begin{align*}
\frac{d}{dt}F(t,q(t,x)){} _{\big| x=q_t^{-1}(z)}&= \partial_tF(t,q(t,x)){} _{\big| x=q_t^{-1}(z)}+ \partial_tq(t,x)\scal\nabla F(t,q(t,x)){} _{\big| x=q_t^{-1}(z)}\\
	&= \partial_tF(t,z)+v(t,z)\scal\nabla F(t,z).
\end{align*}
This formula is meaningful if for each $t,$ the map $x\mapsto q(t,x)$ is injective.
 The operator
\begin{align}\label{eq-04}
\mathrm{D}_t={\partial_t } + {v}\scal \nabla
\end{align}
is sometimes called the convective derivative.

\subsection*{Euler equation}
Let $\XX$ be a bounded domain of $\Rn.$
A fluid in $\XX$ is said to be incompressible if the volume is preserved along the flow. This is equivalent to
\begin{equation*}
\nabla\scal v=0,
\end{equation*}
that is the divergence of the velocity field $v$ vanishes everywhere. If the domain $\XX$ has a boundary $ \partial\XX,$ the  impermeability condition is
$$
n\scal v=0
$$ 
where $n$ is a normal vector to $ \partial\XX. $ From now on, we shall restrict our attention to domains $\XX$ without boundary so that the impermeability condition  is dropped down.
\\
 The Euler equation is Newton's equation of motion
\begin{equation*}
 \mathrm{D}_t v= - \nabla p
\end{equation*}
where $ \mathrm{D}_t v$ is the convective acceleration and  the scalar  pressure field  $p:\ii\times \XX\to\RR$ is part of the solution to be found out with $v$. The force $-\nabla p$ is necessary  for the volume to be preserved as time evolves. The fluid moves from high pressure to low pressure areas. 
 Because of the expression \eqref{eq-04} of the convective derivative, we obtain
 \begin{equation}\label{eq-05}
\left\{\begin{array}{ll}
{\partial_t {v}} + {v}\scal \nabla {v}+ \nabla p=0, & t\ge0, x\in\XX\\
\nabla\scal v=0, & t\ge 0, x\in \XX\\
v(0,\cdot)=v_0, &t=0
\end{array}\right.
\end{equation}
which is the Euler equation of the unknown $(v,p)$ for an incompressible fluid seen as a Cauchy problem with a given initial velocity field $v_0.$

\subsection*{Arnold's flow of diffeomorphisms}

The Cauchy problem \eqref{eq-05} is notoriously difficult  and there  is some hope to understand it a little further by considering a variant  which is closer to a variational approach of  classical mechanics.
Arnold \cite{Arn66} proposed to look at the following fixed  endpoint version of the Cauchy problem \eqref{eq-05}:
\begin{equation}\label{eq-06}
\left\{ \begin{array}{ll}
\partial_t v+v\scal\nabla v+\nabla p=0,& 0\le t\le 1,\\
\nabla\scal v=0,& 0\le t\le 1,\\
q_1[v]= h, &
\end{array}\right.
\end{equation}
where  $\XX$ is  a compact manifold with  no boundary, typically $\XX=\Tn,$
\begin{itemize}
\item
$q_1[v]$ is defined by $q_1[v](x):= \omega_1^x,\ x\in\XX,$  with  
 $ \left\{ \begin{array}{ll}
\dot \omega_t^x=v(t ,\omega_t^x),& 0\le t\le 1\\
\omega_0^x=x, &t=0
\end{array}\right.$;
\item
$h$ belongs to the group $G _{ \vol}:= \left\{ g \textrm{ diffeo: } \mathrm{det\, D}g=1\right\}$ of all volume and orientation preserving diffeomorphisms of $\XX$.
\end{itemize}
This should be regarded informally since it is implicitly assumed in the definition of $q_1[v]$ that the field $v$ admits a unique integral curve for each starting point $x$. 
In fact, the exact  purpose of  \cite{Arn66} is to describe the fluid evolution by means of   pathlines $(q_t(\cdot)) _{ 0\le t\le 1}$ which are seen as trajectories in  $G _{ \vol}$. One can prove that any solution $(q_t(\cdot))_{0\le t\le 1}$ of the  action minimizing  problem
\begin{equation}\label{eq-07}
\IiX  |\partial_t q_t(x)|^2\, dxdt\to \textrm{min}:\quad q_t(\cdot)\in G _{ \vol},\forall 0\le t\le 1,\quad q_0(\cdot)=\Id,\ q_1(\cdot)=h
\end{equation}
 where $h$ is a prescribed element of $G _{ \vol}$, is such that the velocity field $v(t,z)=\partial_t q_t(q_t^{-1}(z))$  is a solution of \eqref{eq-06} for some pressure field $p$. This minimizer is nothing but a geodesic flow on $G _{ \vol}$ with prescribed endpoint positions $\Id$ and $h$. The pressure  $p$ disappears from the picture since  $-\nabla p$ can be seen as the force necessary to maintain the motion on the manifold  $G$ of all diffeomorphisms inside the submanifold $G _{ \vol}$ of volume preserving diffeomorphisms (the orientation is automatically preserved by continuity of the motion).

\subsection*{Brenier's generalized flow}

Solving the geodesic problem \eqref{eq-07} in $ G _{ \vol}$ remains  difficult. Indeed, the only known attempt  is done in \cite{EM70} where a solution is proved to exist for $h$ very close to the identity. Actually there are examples where such geodesics do not exist, see \cite{Shn85,Shn94}. Therefore, Brenier \cite{Bre89}  relaxed \eqref{eq-07} by introducing  a probabilistic representation.
 Brenier's problem consists of minimizing an average kinetic action subject to incompressibility and  endpoint constraints. It is
\begin{equation}\label{eq-08}
E_Q \Iii |\dot X_t|^2\, dt\to \textrm{min};\qquad Q\in\PO: [Q_t=\vol,\forall 0\le t\le1],\  Q _{01}=\pi
\end{equation}
where   $\pi\in\PXX$ is a prescribed bistochastic probability measure, i.e.\ its marginals satisfy  $\pi_0=\pi_1=\vol$ and $\dot X_t( \omega)=\dot \omega_t$ for any absolutely continuous path $ \omega\in\OO$ with generalized time derivative $\dot \omega.$ In the above action functional, it is understood that $\Iii |\dot \omega_t|^2\, dt=\infty$ whenever $\omega\in\OO$ is not absolutely continuous. Therefore   any solution $P$ of \eqref{eq-08} is a path measure charging absolutely continuous paths.  The constraint $(Q_t=\vol,\forall 0\le t\le1)$ reflects the volume preservation.  The prescription that $Q _{01}=\pi$ {  varies among all the possible correlation structures between the initial and final positions with average profile $Q_0=Q_1=\vol.$}  It is a relaxation of  $q_1(\cdot)=h$ as can be seen by taking $\pi(dxdy)=\pi^h(dxdy):=\vol(dx) \delta _{h(x)}(dy).$

 It is proved in \cite{Bre89} that  any $P\in \PO $ such that 
$ \left\{ \begin{array}{l}
P_t=\vol,\forall t \textrm{ and } P _{ 01}=\pi\\
\ddot X_t+\nabla p(t,X_t)=0,\ \forall t,\ P \ae
\end{array}\right.,
$
 for some pressure field $p$, solves the geodesic problem \eqref{eq-08}. Keeping Arnold's point of view, we see that the $\PXX$-valued flow $(P _{ 0t}) _{ 0\le t\le1}$ is the generalized solution of Arnold's geodesic equation \eqref{eq-07}.
In this approach one can recover the velocity field by defining a probability measure $\sigma$ on $[0,1]\times \XX \times \mathbb R^n$,
$$\int _{ \ii\times\XX\times\Rn} f(t,x,v)\, \sigma (dtdxdv):=E_P\int_0^1  f(t, X_t, \dot X_t )\,dt$$
This measure can be considered as a generalized velocity that solves the Euler equation in the sense of DiPerna and Majda, see \cite{Bre99} for details.

 \subsection*{Navier-Stokes equation}
 
 Navier-Stokes equation is a modification of the Euler equation where some viscosity  term is added. Its Newtonian expression is
\begin{equation*}
\mathrm{D}_t v=a\Delta v - \nabla p
\end{equation*}
where  $\Delta v$, the Laplace operator applied to $v,$ represents   a viscosity force, $a>0$.
This equation mixes the acceleration $ \mathrm{D}_t v$ which is easily expressed in Lagrangian coordinates and the viscosity term $\Delta v$ which is easily expressed in terms of Eulerian coordinates. Rewriting everything in  Eulerian terms leads to 
 \begin{equation}\label{eqac-04}
\left\{\begin{array}{ll}
{\partial_t {v}} + {v}\scal \nabla {v}-a \Delta v+ \nabla p=0, & t\ge0, x\in\XX\\
\nabla\scal v=0, & t\ge 0, x\in \XX\\
v(0,\cdot)=v_0, &t=0.
\end{array}\right.
\end{equation}
This is the Navier-Stokes equation of the unknown $(v,p)$ for an incompressible fluid seen as a Cauchy problem with a given initial velocity field $v_0.$

\subsection*{Introducing the Brownian motion}

The presence of the Laplacian in \eqref{eqac-04} strongly suggests that considering Brownian paths instead of regular paths in \eqref{eq-08} might lead us to an approach of the Navier-Stokes equation similar to Brenier's approach to the Euler equation. But one immediately faces the problem of defining the kinetic action $\Iii |\dot X_t|^2/2 \, dt$ since the Brownian sample paths are nowhere differentiable and any discrete approximation of the action diverges to infinity. One is forced to introduce an analogue of the average action $E_Q\Iii |\dot X_t|^2/2 \, dt$ that  appeared in \eqref{eq-08} by considering 
\begin{equation}\label{eq-10}
E_Q \Iii |v^Q_t(X)|^2/2\,dt
\end{equation} 
with a relevant notion of \emph{stochastic velocity} $v^Q_t(X)$ introduced in place   of  the usual velocity $v(t,X_t)=\dot X_t$,  undefined in the present context where the path measure $Q$ charges  Brownian sample paths. A relevant notion of stochastic velocity  was introduced by Nelson in \cite{Nel67}.  
The \emph{forward} stochastic velocity is defined by
\begin{align}\label{eq-11}
\vf^Q_t(X _{ [0,t]}):=\lim _{ h\to 0^+} \frac{1}{h}E_Q(X _{ t+h}-X_t\mid X _{ [0,t]})
\end{align}
and 
its \emph{backward} counterpart by
\begin{align}\label{eq-12}
\vb^Q_t(X _{ [t,1]}):=\lim _{ h\to 0^+} \frac{1}{h}E_Q(X _{ t}-X _{ t-h}\mid X _{ [t,1]}),
\end{align}
provided that $X$ is a $Q$-integrable process and these limits exist in some sense.
The stochastic action \eqref{eq-10} computed  with Nelson's  stochastic velocity $v^Q=\vf^Q$  can be expressed in terms of a relative entropy with respect to the reversible Brownian motion.  

In the whole paper the reference path measure is the law $R$ of the reversible Brownian motion on $\XX= \Tn$ or $\XX=\Rn$ with constant diffusion coefficient $a>0.$

\begin{definition}\label{def-01}
The reversible Brownian path measure $R$ is defined by
\begin{align*}
R=\IX R^x(\cdot)\,\vol(dx)\in\OO
\end{align*}
where for each $x\in\XX,$ $R^x\in\PO$ is the law of $x+ \sqrt{a}\,B$ where $B$ is a standard Brownian motion on $\XX$ starting from $0.$
\end{definition}

Roughly speaking, $R$ is the Wiener measure with diffusion coefficient $a$ on $\XX$ starting from $R_0=\vol\in\PX.$  As $R$ is a reversible Markov measure, its forward and backward velocities are opposite to each other: $\vf^R+\vb^R=0.$ Be aware that we have chosen Nelson's convention when defining the backward velocity $\vb^Q_t(X _{ [t,1]})=-[\vf ^{ (X^*)\pf Q}_{ 1-t}\circ X^*](X _{ [t,1]})$ where $X^*_t:= X _{ 1-t},$ $0\le t\le 1,$ is the time reversed canonical process.
\\
When $\XX=\Rn,$ $R\in\MO$ is an unbounded $ \sigma$-finite measure and when $\XX=\Tn,$ $R\in\PO$ is a probability measure. Note that in any case, the conditioned path measures $R^x\in\PO$ are probability measures.

  Girsanov's theory allows to show that for any $Q\in\MO$ with a finite relative entropy $H(P|R)< \infty,$ there is some predictable vector field $\vf^Q(t,X _{ [0,t]})$ such that $Q$ is the unique solution, among the path  measures with the initial marginal $Q_0$ and  which are absolutely continuous with respect to $R$, of the martingale problem associated with the family of second order differential operators defined for any twice differentiable function $u$   and all $0\le t\le 1$, by
\begin{equation*}
L_tu= \vf^Q(t,X _{ [0,t]})\scal \grad u+ a\Delta u/2.
\end{equation*}
Moreover, we have
\begin{equation}\label{eq-13}
H(Q|R)=H(Q_0|R_0)+ \frac{1}{2a}E_Q\Iii |\vf^Q(t,X _{ [0,t]})|^2\,dt.
\end{equation}
For the details, see for instance \cite{Leo11a}.
\\
These  considerations were put forward a long time ago by Yasue in \cite{Ya83} who introduced the stochastic action \eqref{eq-10} but didn't take advantage of its representation \eqref{eq-13} in terms of the relative entropy, although \eqref{eq-13} is invoked in \cite[p.\,135]{Ya83}.
Since $R$ is reversible with $R_0=R_1=\vol$,  we also obtain
\begin{equation*}
H(Q|R)=H(Q_0|\vol)+\frac{1}{2a}E_Q\Iii |\vf^Q_t|^2\,dt=\frac{1}{2a}E_Q\Iii |\vb^Q_t|^2\,dt+ H(Q_1|\vol)
\end{equation*}
where the stochastic drift fields $\vf^Q$ and $\vb^Q$ given by Girsanov's theory are precisely the forward and backward stochastic velocities of $Q$ properly defined in some $L^2$ spaces. It is therefore natural to address the following entropy minimization problem 
\begin{equation}\label{eq-14}
H(Q|R)\to \textrm{min}; \qquad Q\in\MO: [Q_t=\vol,\forall 0\le t\le1],\  Q _{01}=\pi
\end{equation}
in analogy with Brenier's problem \eqref{eq-08}. Of course, as $Q_0$ and $Q_1$ are prescribed to be the volume measure, we see that
\begin{equation*}
H(Q|R)=\frac{1}{2a}E_Q\Iii |\vf^Q_t|^2\,dt=\frac{1}{2a}E_Q\Iii |\vb^Q_t|^2\,dt,
\end{equation*}
strengthening  the analogy with  \eqref{eq-08}.

\subsection*{Bredinger's problem}

Recall that the dynamical version of the  Schr\"odinger problem amounts to minimize the relative entropy 
\begin{equation}\label{eq-15}
H(Q|R)\to \textrm{min};\qquad Q\in\MO: Q_0=\mu_0, Q_1=\mu_1
\end{equation}
among all the path  measures $Q$ such that the initial and final marginals $Q_0$ and $Q_1$ are prescribed to be equal to given  measures $\mu_0$ and $\mu_1\in\MX.$  For more details on this convex optimization problem see \cite{Foe85,FG97,Leo12e}.
As  Problem \eqref{eq-14} is an hybrid of Brenier's problem \eqref{eq-08} and Schrödinger's problem \eqref{eq-15}, we call it the Bredinger problem. We introduce the following extension of \eqref{eq-15}:

\begin{equation}\label{eq-bdg}
H(Q|R)\to \textrm{min};\qquad Q\in\MO: (Q_t=\mu_t,\  \forall t\in \TT ),\quad Q _{01}=\pi
\tag{Bdg}
\end{equation}
with $ \TT $ a measurable subset of $\ii$ and $( \mu_t) _{ t\in \TT }$ a prescribed set of nonnegative  measures on $\XX.$ It is a slight   extension of \eqref{eq-14} where the state space $\XX$ and the prescribed marginals $( \mu_t) _{ t \in \TT }$ are general. We still call the extension \eqref{eq-bdg} of \eqref{eq-14} the Bredinger problem.
\\
Of course, for this problem to admit a solution, it is necessary that
$H(\mu_t|\vol)<\infty$ for all $t\in \TT $,  $H(\pi|R _{01})<\infty$ (recall \eqref{eq-03}) and $\left\{\begin{array}{lcl}
\pi_0&:=&\pi(\cdot\times\XX)=\mu_0,\\ 
\pi_1&:=&\pi(\XX\times\cdot)=\mu_1.
\end{array}
\right.
$

\begin{proposition}
Problem \eqref{eq-bdg} admits a solution if and only if there exists some $Q\in\MO$ such that $Q_t= \mu_t$ for all $t\in\TT,$ $Q _{ 01}=\pi$ and $H(Q|R)< \infty.$ In this case, the solution $P$ is unique.
\end{proposition}
\begin{remark}\label{ext1}
It can be checked without difficulty that this result is also valid when the state space $\XX$ 
 is a stochastically complete Riemannian manifold with smooth boundary, $R$ is the reversible Brownian
 path measure with initial marginal $R_0={\rm vol}$.
\end{remark}

\subsection*{Fundamental example on the torus}

 As a basic important example, we take $R\in\PO$  the  Wiener measure on the flat torus $\XX=\mathbb{T}^n$ with initial marginal $R_0=\vol:$ the normalized volume measure (so that $R$ is reversible), $ \mu_t=\vol$ for all $t\in \TT =\ii$  and $\pi$ any bi-stochastic measure, i.e.\ such that $\pi_0=\pi_1=\vol.$ In this setting, \eqref{eq-bdg} becomes \eqref{eq-14} which is as close as possible to Brenier's problem \eqref{eq-08}. 
 \\
It is proved in Corollary \ref{res-07} that in this precise setting, Bredinger's problem \eqref{eq-14} admits a solution if and only if $ \pi $ is such that $H( \pi|R _{ 01})< \infty.$
 
 \begin{remark}\label{ext2}
  In the context of Remark~\ref{ext1}, the fundamental example extends to a compact manifold $\XX$ on which acts transitively
  a compact group of isometries with bi-invariant metric. 
 \end{remark}

 \subsection*{A simplified problem}
 
It will be convenient to consider the easy version of the Bredinger problem \eqref{eq-bdg} with a finite set $ \TT = \left\{t_1,\dots,t_K\right\} $:
\begin{equation}\label{eq-16}
H(Q|R)\to \textrm{min};\qquad Q\in\MO: (Q _{ t_k}=\mu _{ t_k},  0< t_1< \cdots< t_K< 1),\quad Q _{01}=\pi
\end{equation}
where only a finite number of marginal constraints are prescribed.

\section{Duality}\label{sec-dual}

The duality of Brenier's problem was studied in \cite{Bre93}. The present section is devoted to the dual problem of Bredinger's problem. In contrast with Brenier's problem which is affine, the strict convexity of Bredinger's problem allows for a rather standard treatment based on general convex analysis in infinite dimensional spaces stated below at Theorem \ref{res-02}. The main results of the section are the dual equality of Proposition \ref{res-10} holding under  weak hypotheses, and a characterization of the solution of Bredinger problem stated at Corollary \ref{res-03}, valid under restrictive regularity assumptions.

The dual problem of \eqref{eq-bdg} is stated at \eqref{eq-52}. Its unknown are a pressure scalar field $p:\iX\to \RR$ in duality with the incompressibility  constraint and a function $\eta:\XXX\to\RR$ in duality with the endpoint constraint $\pi.$

\subsection*{Dual equality}

Let us denote for all $x\in\XX,$ $R^x:=R(\cdot\mid X_0=x)\in\PO$ and $R ^{ \mu_0}=\IX R^x(\cdot)\, \mu_0(dx)\in\MO.$ They describe respectively the reference kinematics   starting from $x$ or from the initial distribution $R ^{ \mu_0}_0= \mu_0.$ By \eqref{eq-02} with $f=X_0,$ 
\begin{align}\label{eq-17}
H(Q|R)=H( \mu_0|R_0)+H(Q| R ^{ \mu_0})
\end{align}
for all $Q\in\MO$ such that $Q_0= \mu_0.$ Therefore, as soon as $H( \mu_0|R_0)< \infty$ (this is verified when \eqref{eq-bdg} admits a solution), it is equivalent to solve the modified Bredinger problem
\begin{align}\label{eq-bdgm}
H(Q|R ^{ \mu_0})\to \textrm{min};\qquad Q\in\MO: Q_t= \mu_t,\ t\in\TT,\quad Q _{ 01}= \pi,
\tag{Bdg$ ^{ \mu_0}$}
\end{align}
or \eqref{eq-bdg}. The problems \eqref{eq-bdg} and \eqref{eq-bdgm} share the same solution but their values differ from the  quantity $H( \mu_0|R_0)$ that only depends on the prescribed data $R$ and $ \mu.$ As far as duality is concerned, it will be a little bit more comfortable to consider \eqref{eq-bdgm} rather than \eqref{eq-bdg}.

Take $ \alpha$  a probability measure on $ \TT $ and consider the following  weakening of \eqref{eq-bdg}:
\begin{equation}\label{eq-18}
H(P|R)\to \textrm{min};\qquad P\in\MO: (P_t=\mu_t,\  \textrm{ for $ \alpha$-almost all } t\in \TT ),\quad Q _{ 01}= \pi.
\end{equation}
For instance, one may take $ \alpha=K ^{ -1}\sum _{ 1\le k\le K} \delta _{ t_k}$ for the Bredinger problem \eqref{eq-16}. Choosing $ \alpha= \mathrm{Leb} _{ \ii}$ leads to  $Q_t=\mu_t,$ for \emph{almost} all $t\in\ii$ which is a slight weakening of $Q_t=\mu_t,$ $\forall t\in\ii$ in the original Bredinger problem \eqref{eq-14}. Nevertheless, the following result holds.

\begin{lemma}\label{res-06}
Assume that $ \alpha$ has a full support, i.e.\ 
$
\supp \alpha= \TT ,
$
and consider the following statements:
\begin{enumerate}[(i)]
\item
 \eqref{eq-bdg} admits a solution;
\item
$t\in \TT \mapsto \mu_t\in\PX$ is weakly  continuous on $\TT$;
\item
 \eqref{eq-bdg}  is equivalent to   \eqref{eq-18}.
\end{enumerate}
We have: $(i)\implies (ii)\implies (iii).$
\end{lemma}
\begin{proof}
Since for any $P\in\MO,$ $t\in\TT\mapsto P_t\in\PX$ is weakly  continuous (this follows from the  continuity of the sample paths), it is necessary for \eqref{eq-bdg} to admit a solution such that 
$t\mapsto \mu_t$ is also weakly  continuous.  In such a case, under the assumption that $\supp \alpha= \TT ,$ the constraint $(P_t=\mu_t,\  \textrm{for $ \alpha$-almost all } t\in \TT )$ is equivalent to $(P_t=\mu_t,\  \forall t\in \TT ).$
\end{proof}

For the moment, it is assumed that the constraint $ \mu= (\mu_t) _{ t\in\TT}$ is a flow of \emph{probability measures} and $\pi\in\PXX$ is also a probability measure. In particular, this implies that \eqref{eq-bdg} is
\begin{equation*}
H(Q|R)\to \textrm{min};\qquad Q\in\PO: (Q_t=\mu_t,\  \forall t\in \TT ),\quad Q _{01}=\pi
\end{equation*}
where $Q$ lives in $\PO$ rather than in $\MO.$

\begin{hypotheses}\label{hyp-02}
\begin{enumerate}[(a)]
\item
The constraint $ \mu= (\mu_t) _{ t\in\TT}$ is a flow of \emph{probability measures}.

\item
The mapping $t\in \TT \mapsto \mu_t\in\PX$ is weakly  continuous on $\TT$ .  

\item
We choose $ \alpha\in \mathrm{P}(\TT)$ such that
$ \supp\alpha=\TT.$
\end{enumerate}
\end{hypotheses}
It follows from Lemma \ref{res-06} that we can assume Hypothesis \ref{hyp-02}-(b) almost without any loss of generality, so that \eqref{eq-bdg} and \eqref{eq-18} are equivalent. 

The  aim of this section is to prove a dual equality for \eqref{eq-bdg}.

 Before  stating it at Proposition \ref{res-10}, we must introduce some notions and notation.
The value function of problem \eqref{eq-bdgm} is denoted  by
\begin{align*}
J( \mu, \pi):=\inf \eqref{eq-bdgm}
	=\inf \left\{ H(Q|R ^{ \mu_0}); Q\in\PO: Q_t= \mu_t,\ \forall t\in\TT,
		\ Q _{ 01}= \pi\right\}.
\end{align*}

We denote respectively $\BTX$ and $\BXX$ the  spaces of bounded measurable functions of $\TX$ and $\XXX.$

\begin{proposition}\label{res-10}
For any  $\pi\in\PXX$, $ \mu\in\PX ^{ \TT}$ and  $ \alpha\in \PT$ satisfying the Hypotheses \ref{hyp-02} , we have
\begin{align*}
\inf \eqref{eq-bdg}=H( \mu_0|R_0)+J( \mu, \pi).
\end{align*}
Moreover, for any class of functions  $ \mathcal{A} $ dense in $B(\TX)\times B(\XXX)$ with respect to the pointwise convergence, the dual equality is
\begin{align}\label{eq-52}
J( \mu, \pi)
	=\sup _{ (p,\eta)\in \mathcal{A}} \Big\{ &\left\langle p,\ma \right\rangle + \left\langle \eta, \pi \right\rangle\nonumber \\
	&- \IX \log E _{ R^x} \exp\left( \IT p(t,X_t)\, \alpha(dt)+ \eta(x,X_1)\right) \, \mu_0(dx)\Big\},
\end{align}
where we denote $\ma(dtdx):=  \alpha(dt)\mu_t(dx).$
 \end{proposition}
Of course, the first identity is a direct consequence of \eqref{eq-17}.\\
The remainder of the present section is devoted to the proof of the dual equality.

\subsection*{An abstract duality result}

We begin stating  an already known abstract result about convex duality. We shall apply it later  to derive the dual problem of \eqref{eq-bdg} and some basic related relations.
In these lines, we rely on general results about convex duality as presented   for instance in the lecture notes \cite{Leo-LN}. Let $U$ and $V$ be two Hausdorff locally convex topological vector spaces with respective topological dual spaces $U'$ and $V'$.
Let us consider the following minimization problem
\begin{equation}\label{eq-19}
I(\ell)\to \textrm{min}; \quad \ell\in U':\ T\ell=k_o
\tag{$\mathcal{P}$}
\end{equation}
where $I$ is a convex $(-\infty,\infty]$-valued function on $U'$, $T: U'\rightarrow V'$ is a linear operator and $k_o\in V'.$
We assume that the algebraic adjoint operator $T^*$ of $T$ satisfies 
$	
    T^*(V)\subset U,
$	
so that one can write $[T^* v](\ell)=\langle
T^*v,\ell\rangle_{U,U'}=\langle v,T\ell\rangle_{V,V'}.$ It follows that
the diagram
\begin{equation*}
 \begin{array}{ccc}
\Big\langle\ U & , & U' \ \Big\rangle \\
T^\ast  \Big\uparrow & & \Big\downarrow
 T
\\
\Big\langle\ V & , & V'\ \Big\rangle
\end{array}
\end{equation*}
is meaningful. The associated dual problem is
\begin{equation*}
   \langle v,k_o \rangle -I^*(T^*v)\to \textrm{max};\quad v\in V
\tag{$\mathcal{D}$}
\end{equation*}
where
\begin{equation*}
I^*(u):= \sup _{ \ell\in U'} \left\{ \left\langle \ell,u \right\rangle -I(\ell)\right\} ,\quad u\in U
\end{equation*}
is the convex conjugate of $I$ with respect to the duality $ \left\langle U,U' \right\rangle .$

\begin{theorem}\label{res-02}
Let us suppose that the following assumptions on $I$  and $T$ hold:
\begin{itemize}
    \item[(i)]  $I$ is a convex $\sigma(U',U)$-\lsc\ function such that  $\inf I>-\infty;$
       \item[(ii)]  there exists an open neighbourhood $N$ of $0$ in $U$ such that $\sup_{u\in N}I^*(u)<+\infty;$
          \item[(iii)]  $T^*V\subset U.$
\end{itemize}
Then, the following assertions are verified.
\begin{enumerate}
\item
If $\inf(\mathcal{P})<+\infty,$ the primal problem $(\mathcal{P})$ admits at least a solution and if $I$ is strictly convex, this solution is unique.

\item
The dual equality $ \inf(\mathcal{P})= \sup(\mathcal{D})$ holds. That is
\begin{equation*}
    \inf\{I(\ell); \ell\in U': T\ell= k_o\}=\sup_{v\in V}\big\{ \langle k_o,v \rangle 
    -I^*(T^*v)\big\}\in (-\infty,+\infty].
\end{equation*}

\item
 The primal and dual problems are attained at $\bar{\ell}$ and
    $\bar{v}$ respectively if and only if the following  relations hold: $ \left\{ \begin{array}{l}
    T\bar{\ell}=k_o,\\
            \bar{\ell}\in\partial I^*(T^*\bar{v}),
    \end{array}\right.$
    where \\
    $ \partial I^*(T^*\bar{v}):= \left\{ \ell\in U'; I^*(T^*\bar{v}+w)\ge I^*(T^*\bar{v})+ \left\langle \ell, w \right\rangle _{ U',U}, \forall w\in U\right\}$ denotes  the subdifferential of $I^*$ at $T^*\bar{v}$.
\end{enumerate}
\end{theorem}

\subsection*{Proof of Proposition \ref{res-10}}

We apply Theorem \ref{res-02} to the primal problem \eqref{eq-bdg}.  We first chose a relevant set of vector spaces and objective functions $U,U^*, I$ and $I^*.$ Then, we look at the constraint operators $T$ and $T^*$. Finally,  Theorem \ref{res-02} is applied in this setting.

\subsubsection*{The objective functions $I$ and $I^*$}

Let us denote $\BO$ the space of all bounded measurable functions on $\OO$ and equip it with the uniform norm $\|u\|:=\sup _{ \omega\in\OO}| u( \omega)|,
$ $u\in\BO.$
Its topological dual space is denoted by $\BO'.$ The convex  function
 \begin{equation}\label{eq-21}
\Theta(u):=\IX\log (E _{ R^x} e^u)\, \mu_0(dx),\quad u\in\BO
\end{equation}
is well defined on $\BO$ because for any $u\in\BO$ 
\begin{align}\label{eq-22}
|\log E _{ R^x}e ^{ u}|\le \|u\|,\quad \forall x\in\XX,
\end{align}
 implying  
\begin{align}\label{eq-23}
|\Theta(u)|\le\|u\|  < \infty.
\end{align}

Comparing \eqref{eq-bdg}  with \eqref{eq-19} and taking Lemma \ref{res-01} below into account, we see that a good framework to work with is: $(U,\|\cdot\|)=(\BO,\|\cdot\|)$ and 
\begin{equation*}
I(Q):= \Theta^*(Q)= \sup _{ u\in\BO} \left\{ \left\langle Q,u \right\rangle 
	-\IX\log (E _{ R^x} e^u)\, \mu_0(dx)\right\} ,\quad Q\in\BO',
\end{equation*}
 the convex conjugate of $ \Theta.$

\begin{lemma}\label{res-01}
For any $Q\in\BO',$ 
$
I(Q)= \left\{ \begin{array}{ll}
H(Q|R ^{ \mu_0}),& \textrm{if }Q\in\PO \textrm{ and }Q_0= \mu_0,\\
+ \infty,& \textrm{otherwise.}
\end{array}\right.
$
\end{lemma}
\begin{proof}
This proof follows the line of the proof of \cite[Lemma 4.2]{Leo12a}. 
\\
 Let  $Q\in\BO'$  be such that $I(Q)< \infty.$  
\begin{enumerate}[(a)]
\item
Let us show that $Q\ge0.$ Take $u\in \BO,$ $u\ge0.$ As for all $a \le0,$ $ \Theta(au)\le 0,$  we get
\begin{equation*}
 I(Q)\ge \sup _{ a\le0}\{a \left\langle Q,u \right\rangle - \Theta(au)\}
 	\ge \sup _{ a\le 0}a \left\langle Q,u \right\rangle 
	= \left\{ \begin{array}{ll}
	0,& \textrm{if }\left\langle Q,u \right\rangle\ge 0 \\
	+ \infty,& \textrm{otherwise}.
	\end{array}\right.
\end{equation*}
Hence, $I(Q)< \infty$ implies that $ \left\langle Q,u \right\rangle \ge 0,$ for all $u\ge 0,$ which is the desired result.

\item
Let us show that $Q$ is a  positive measure. For a positive element of $\BO'$ to be a measure it is sufficient (and necessary) that it is $ \sigma$-additive. This means that for any decreasing sequence $(u_n) _{ n\ge 0}$ of measurable bounded functions such that  $\lim _{ n\ge0} u_n( \omega)=0$ for all $ \omega\in \Omega,$  we have 
$$
\lim _{ n\to \infty} \left\langle Q,u_n \right\rangle =0.
$$ 
Let $(u_n) _{ n\ge0}$ be such a sequence. By dominated convergence,  for all $a\ge 0,$ we obtain $\lim _{ n\to \infty} \Theta(a u_n)=0.$ Therefore,
\begin{equation*}
I(Q)\ge \sup _{ a\ge0}\limsup _{ n\to \infty}\left\{a \left\langle Q,u_n\right\rangle - \Theta(au_n)\right\} 
	= \sup _{ a\ge0}a\limsup _{ n\to \infty} \left\langle Q,u_n\right\rangle 
\end{equation*}
and $I(Q)< \infty$ implies that $\limsup _{ n\to \infty} \left\langle Q,u_n\right\rangle \le 0.$ Since, we already know that $Q\ge0,$ this gives the desired result: $\lim _{ n\to \infty} \left\langle Q,u_n \right\rangle =0$.

\item
Let $Q\in\MO.$ Taking $u=f(X_0)$ in $\sup_u$ gives
\begin{align*}
I(Q)\ge \sup _{ f} \left\langle f,Q_0- \mu_0 \right\rangle 
	= \left\{ \begin{array}{ll}
	0,& \textrm{if }Q_0= \mu_0,\\
	\infty,& \textrm{otherwise.}
	\end{array}\right.
\end{align*}
Hence, $I(Q)< \infty$ implies that $Q_0= \mu_0\in\PX$ which in turns implies that $Q$ is also a probability measure.   
\\
In this case, $Q=\IX Q^x(\cdot)\, \mu_0(dx)$ and
\begin{align*}
I(Q)
	&=\sup _{ u}\IX \left(E _{ Q^x}u-\log E _{ R^x}e^u\right) \, \mu_0(dx)\\
	&\overset{(i)}= \sup _{ k\ge 1}\IX \sup _{ u^x: |u^x|\le k} \left(E _{ Q^x}u^x-\log E _{ R^x}e^{u^x}\right) \, \mu_0(dx)\\
	&= \sup _{ k\ge 1}\IX \sup _{ v: |v|\le k} \left(E _{ Q^x}v-\log E _{ R^x}e^{v}\right) \, \mu_0(dx)\\
	&\overset{(ii)}=\IX H(Q^x|R^x)\, \mu_0(dx)\\
	&\overset{ \eqref{eq-17}}=H(Q|R ^{ \mu_0})
\end{align*}
which is the announced result. Let us give some precisions about this series of identities.
At (i), we used the notation $u^x$ for the restriction of $u$ to $\OO^x:= \left\{ X_0=x\right\} \subset\OO$. Note that the inversion of $\sup_u$ and $\IX$ is valid since any function $u\in\BO$ can be identified with a measurable kernel $(u^x\in B(\OO^x),\ x\in\XX)$ by $u=u ^{ X_0}.$ Identity (ii) follows from a standard variational representation of the relative entropy  of probability measures
 (see  \cite[Appendix]{GL10} for instance) combined with Beppo-Levi's monotone convergence theorem.  
 \end{enumerate}
This completes the proof of the lemma. 
\end{proof}

 Let us compute $I^*= \Theta ^{ **}$ defined by 
$I^*(u):=\sup _{ Q\in\BO'} \left\{ \langle Q,u\rangle -I(Q)\right\},$ $ u\in \BO.
$
As $ \Theta$ is convex (by Hölder's inequality) and lower $ \sigma(\BO,\BO')$-semicontinuous (by Fatou's lemma, it is lower $\|\cdot\|$-semicontinuous and since it is convex, it turns out to be weakly semicontinuous), it is equal to its convex biconjugate. This means that 
$$  
I^*=\Theta ^{ **}= \Theta.
$$

\subsubsection*{The constraint operators $T$ and $T^*$}

 For any $Q\in \BO',$ we define 
 $
 \widetilde{Q}\in B(\TX)'$ by
 $$ \langle \widetilde{Q},p \rangle =\IT \left\langle Q,p(t,X_t) \right\rangle \,\alpha(dt),
\quad p\in B(\TX).
$$
 Clearly, when $Q$ belongs to $\PO,$ $\widetilde{Q}$ is the  measure defined by $\widetilde{Q}(dtdx)=\alpha(dt)Q_t(dx)$. Hence, defining $\ma (dtdx):= \mu_t(dx)\alpha(dt)$, we see that  $\widetilde{Q}=\ma $   is equivalent to $Q_t= \mu_t$ for $ \alpha$-almost all $t\in \TT .$
\\
For any $Q\in\BO'$, we  define $Q _{ 01}\in B(\XXX)'$  by: $ \left\langle Q _{ 01}, \eta \right\rangle = \left\langle Q,\eta(X_0,X_1) \right\rangle ,$ $\forall \eta\in B(\XXX).$
\\
Putting everything together, the constraint operator is defined by
\begin{equation*}
TQ:=(\widetilde{Q}, Q _{ 01 })\in B(\TX)'\times B(\XXX)',\quad Q\in\BO'
\end{equation*}
and the full constraint of \eqref{eq-18} writes as 
$$
TQ=(\ma , \pi),\quad Q\in\BO'.
$$
It is time to identify the topological space $V$ as $V=B(\TX)\times B(\XXX)$ equipped with the uniform norm $\|\cdot\| _{ \TX}\oplus \|\cdot\| _{ \XXX},$ so that its topological space $V'=B(\TX)'\times B(\XXX)'$ contains $ \mathrm{M}_b(\TX)\times \mathrm{M}_b(\XXX).$ 
\\
Let us compute the adjoint $T^*$ of $T$. For any $Q\in\BO',$ $p\in B(\TX)$ and $\eta\in B(\XXX),$ 
\begin{equation*}
\left\langle Q,T^*(p,\eta) \right\rangle 
	= \left\langle TQ,(p,\eta) \right\rangle 
	= \big\langle \widetilde{Q},p \big\rangle + \left\langle Q _{ 1 },\eta \right\rangle 
	= \left\langle Q, \IT p(t,X_t)\,\alpha(dt) + \eta(X_0,X_1)\right\rangle.
\end{equation*}
Consequently, we see that
\begin{align}\label{eq-24}
T^*(p,\eta)=\IT p(t,X_t)\,\alpha(dt) + \eta(X_0,X_1).
\end{align}

\subsubsection*{The dual problem}

We gathered all the ingredients to see that the dual problem of \eqref{eq-18} is 
\begin{alignat}{1}\label{eq-25}
\ITX p\, d\ma+ \IXX \eta\,d\pi-\IX\log E _{ R^x} &\exp \left( \IT p(t,X_t)\,\alpha(dt) + \eta(x,X_1)\right) \, \mu_0(dx)
		\to \textrm{max};\nonumber\\
&\hskip 3cm p\in B(\TX), \ \eta\in B(\XXX).
\end{alignat}

\begin{proof}[Proof of Proposition \ref{res-10}]
It remains to verify the assumptions of Theorem \ref{res-02} to obtain
 \begin{align*}
J( \mu, \pi)
	=\sup _{ (p,\eta)\in B(\TX)\times B(\XXX)} &\Big\{ \left\langle p,\ma \right\rangle + \left\langle \eta, \pi \right\rangle \\
	&- \IX \log E _{ R^x} \exp\left( \IT p(t,X_t)\, \alpha(dt)+ \eta(x,X_1)\right) \, \mu_0(dx)\Big\}.
\end{align*}

\begin{enumerate}[(i)]
\item
Being a convex conjugate, $I= \Theta^*$ is convex and \lsc\ with respect to  $ \sigma(\BO',\BO)$.
\item
We see with \eqref{eq-23}, that the  function $I^*= \Theta$ is such that on the  unit ball $N= \left\{u\in\BO; \|u\|\le1\right\} $, we have: $\sup _{ u:\|u\|\le 1} \Theta(u)< \infty.$

\item
It is clear that for any  $p$ in $B(\TX)$ and  $\eta$ in $B(\XXX),$ $T^*(p,\eta)=\IT p(t,X_t)\,\alpha(dt) + \eta(X_0,X_1)$ is  in $\BO.$
\end{enumerate}
So far we have proved Proposition \ref{res-10} in the special case where $ \mathcal{A}=B(\TX)\times B(\XXX).$ The extension to the case where $ \mathcal{A}$ is pointwise dense in $B(\TX)\times B(\XXX)$ follows easily from an approximation argument combined with the dominated convergence theorem using \eqref{eq-22}.
This completes the proof of Proposition \ref{res-10}.
\end{proof}

\begin{remark}
We have chosen a strong topology on $U=\BO$ to insure the estimate $\sup _{ u\in N} \Theta(u)<\infty$ at (ii) above. This  explains why  Lemma \ref{res-01} is needed.
\end{remark}

\subsection*{Regular solutions}

As a byproduct of this proof, Theorem \ref{res-02} leaves us with  the following

\begin{corollary} \label{res-03}\ 
Assume that Hypothesis \ref{hyp-02} holds.
\begin{enumerate}
\item
If $\inf \eqref{eq-bdg}<+\infty,$ the primal problem \eqref{eq-bdg} admits a unique solution.

\item
Let $P\in\PO$, $p\in B(\TX)$ and $\eta\in  B(\XXX)$.  
 Both the primal problem \eqref{eq-bdg}  and the dual problem \eqref{eq-25} are  attained respectively  at $P$ and
    $(p,\eta)$ if and only if the constraints 
    $P_t=\mu_t, \forall t\in \TT $ and $P _{ 01}= \pi$
    are satisfied and
\begin{equation}\label{eq-26}
P= \frac{d \mu_0}{dR_0}(X_0)\exp \left(\eta(X_0,X_1) - \mathcal{Q}( p,\eta)(X_0)+ \IT p(t,X_t)\,\alpha(dt)\right) \,R
\end{equation}
where 
\begin{align*}
\mathcal{Q} (p, \eta)(x):=\log E _{ R^x}\exp \Big( \IT p(t,X_t)\,\alpha(dt)+\eta(x,X_1)\Big),
\quad x\in\XX,\ R_0\as
\end{align*}
\end{enumerate}
\end{corollary}

\begin{proof}
Statement (1) is a direct application of Theorem \ref{res-02}-(1) and 
the only thing to be checked for (2)   is the computation of the subdifferential $ \partial I^*(T^*\bar v)$ of Theorem \ref{res-02}-(3) in the present setting. With $I^*= \Theta$ given at \eqref{eq-21}, we obtain for any $u\in \BO,$
\begin{align*}
 \Theta'(u)&= e^u\ \IX R^x(\cdot) \frac{\mu_0(dx)}{E _{ R^x}e^u}
 	=\exp \left( u-\log E _{ R ^{ X_0}}e^u\right)\,\IX R^x(\cdot) \, \mu_0(dx)\\
	&= \exp \left( u-\log E _{ R ^{ X_0}}e^u\right)\, R ^{  \mu_0}
	= \frac{d \mu_0}{dR_0}(X_0) \exp \left( u-\log E _{ R ^{ X_0}}e^u\right)\, R.
\end{align*}
We conclude with \eqref{eq-24}, i.e.\ $T^*\bar v =T^*(p, \eta)=\IT p(t,X_t)\,\alpha(dt)+\eta(X_0,X_1).$
\end{proof}

\section{General shape of the solution}\label{sec-shape}

Assuming the dual attainment $p\in B(\TX)$ and $\eta\in  B(\XXX)$ as in   Corollary \ref{res-03}-(2) is very restrictive. In general, even if the solution $P$ writes as \eqref{eq-26}, $p$ and $\eta$ might be unbounded or even might take the value $-\infty$ at some places.  

It is proved in \cite{BL16} that, provided that the reference path measure $R$ is Markov, the solution $P$ to Bredinger's problem is a reciprocal path measure. A representation of the Radon-Nikodym derivative $dP/dR$ is also obtained. In order to state these results below at Theorem \ref{res-08}, it is necessary to recall the definitions of  Markov and reciprocal measures and also of additive functional.

First of all, we need to make precise what a conditionable path measure is.

\begin{definitions}[Conditionable path measure]\ 
\begin{enumerate}
    \item A positive measure  $Q\in\MO$ is called a path measure.

    \item  The path measure $Q\in\MO$ is said to  be \emph{conditionable} if for all $t\in\ii,$  $Q_t$
is a $\sigma$-finite measure on $\XX.$
\end{enumerate}
\end{definitions}

It is shown in \cite{Leo12b} that for any  conditionable path measure $Q\in\MO,$ the conditional expectation $E_Q(\cdot\mid X_\mathcal{T})$ is well-defined for any $\mathcal{T}\subset\ii.$ This is the reason for this definition.

\begin{definition}[Markov measure] 
A path measure  $Q$ on  $\Omega$ is said to be  \emph{Markov}  if it is conditionable and 
 if for any $t\in[0,1]$ and for any events $A\in \sigma(X_{[0,t]}), B\in \sigma(X_{[t,1]})$
    \begin{equation*}
    Q(A\cap B\mid X_t)=Q(A\mid X_t)Q(B\mid X_t), \quad Q\ae  
    \end{equation*}
\end{definition}
This means that, knowing the present state  $X_t$, the future and  past   informations  $ \sigma(X_{[t,1]})$  and  $ \sigma(X_{[0,t]})$, are  $Q$-independent.  

\begin{definition}[Reciprocal measure] \label{def-01bis}
A path measure  $Q$ on  $\Omega$ is called \emph{reciprocal} if it is conditionable and for any times  $0\le s\le u\le 1$  and  any events  $A\in \sigma(X_{[0,s]}) , B\in \sigma(X_{[s,u]}), C \in \sigma(X_{[u,1]}) $,

\begin{figure}
\includegraphics{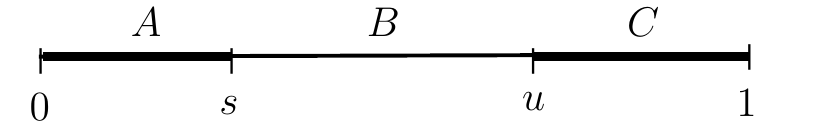}
\caption{}\label{fig-01}
\end{figure}

see Figure \ref{fig-01},
    \begin{equation*}
    Q(A\cap B\cap C \mid X_s, X_u)=Q(A \cap C\mid X_s, X_u)Q(B\mid X_s, X_u)\quad  Q\ae  
    \end{equation*}
\end{definition}

This property  states that under $Q$, 
given the knowledge of the canonical process at  both times  $s$ and
$u$, the events ``inside'' $[s,u]$ and those ``outside'' $(s,u)$  are conditionally independent.
It is clearly  time-symmetric.

\begin{remarks}\ We recall basic relations between the Markov and reciprocal properties.
\begin{enumerate}[(a)]
\item
Any Markov measure is reciprocal, but there are reciprocal measures that are not Markov.
\item
For any reciprocal measure $Q,$ the conditional path measures $Q(\cdot\mid X_0)$ and $Q(\cdot\mid X_1)$ are Markov, $Q\ae$
\end{enumerate}
\end{remarks}

\begin{definition}[Additive functional]
A measurable function $A _{ [0,1]}:\OO\to [- \infty, \infty)$ is said to an $R$-additive functional if for any  finite partition $\bigsqcup _k I_k=[0,1]$ of the  time interval $[0,1]$ with intervals $I_k,$ we have
\begin{equation*}
A _{ [0,1]}=\sum _{ k }A _{ I_k},\quad R\ae
\end{equation*}
where for all $k,$ $A _{ I_k}=A_k(X _{ I_k})$ is $ \sigma(X _{ I_k})$-measurable. 
\end{definition}
With some abuse of notation, we shall write: $A _{ [0,1]}=A=A(X _{ [0,1]}) =\sum_k A(X _{ I_k}).$

We are now ready to state Theorem \ref{res-08} and its hypotheses.

\begin{assumption}\label{ass-01}(Strong irreducibility).\ 
The reference measure $R$ is Markov and it admits a transition density $r$ defined by for all $0\le s<t\le 1$ by
\begin{align*}
R(X_t\in dy\mid X_s=x):=r(s,x;t,y)\, R_t(dy),\quad \forall x\in\XX,\ R_s\ae
\end{align*}
 which is positive in the sense that
\begin{align*}
r(s,x;t,y)>0,\quad \forall (x,y), R_s\otimes R_t\ae,\quad \forall 0\le s<t\le 1.
\end{align*}
\end{assumption}

\begin{theorem}\label{res-08}
Under the above Assumptions \ref{ass-01}, the following assertions hold true.

\begin{enumerate}[(a)]
\item
The solution $P$ of  \eqref{eq-bdg}, if it exists, is  reciprocal and 
\begin{equation}\label{eq-28}
P=\exp (A(X _{ \TT})+ \eta(X_0,X_1))\, R
\end{equation}
for some $[- \infty, \infty)$-valued $ \sigma(X _{ \TT})$-measurable additive functional $A(X _{ \TT})$ and some measurable function $\eta:\XXX\to[- \infty, \infty),$ with the convention that $\exp( - \infty)=0.$

\item
We consider  the prescribed marginals $ \mu _{ t_k}\in\PX,$ $1\le k\le K$ at times  $0\le t_1<\cdots t_K\le1$ and the prescribed endpoint marginal $\pi\in \PXX.$
If the Bredinger problem
\begin{align}\label{eq-29}
H(P|R)\to \mathrm{min};\quad P\in\PO: P _{ t_k}= \mu _{ t_k},\ 1\le k\le K,\ P _{ 01}=\pi,
\end{align}
is such that $\inf \eqref{eq-29}< \infty,$ its unique solution $P$ is reciprocal and writes as
\begin{align}\label{eq-30}
P=\exp \left(\sum _{ 1\le k\le K} \theta _{ t_k}(X _{ t_k})+ \eta(X_0,X_1)\right)\,R
\end{align}
for some measurable functions $ \theta _{ t_k}:\XX\to[- \infty, \infty), 1\le k\le K$ and some measurable function $\eta:\XXX\to[- \infty, \infty).$
\end{enumerate}
\end{theorem}

\begin{proof}
See \cite{BL16}.
\end{proof}
\begin{remark}\label{ext3}
  Following  Remark~\ref{ext1}, it can be proved that Theorem~\ref{res-08} extends to a stochastically complete 
  Riemannian manifold~$\XX$ with $R$ the reversible Brownian path measure having marginal $R_0={\rm vol}$.
 \end{remark}

We see that \eqref{eq-26} has the desired shape \eqref{eq-28} with\begin{align}\label{eq-31}
A(X _{ \TT})=\int _{ \TT} p(t,X_t)\, \alpha(dt)
\end{align}
and \eqref{eq-31} holds true for the solution $P$ of Bredinger problem under the hypotheses of  Theorem \ref{res-08}-(b) where only finitely many time marginal constraints are considered.

\section{Kinematics of regular solutions}\label{sec-regular}

In this section, the reference path measure $R$ is the law of the reversible Brownian motion with diffusion coefficient $a>0$ on $\XX=\Tn$ or $\XX=\Rn,$ see Definition \ref{def-01}.

\begin{definition}[Regular solution of Bredinger's problem]\label{def-02}
In view of \eqref{eq-26} in Corollary \ref{res-03} and \eqref{eq-30} in Theorem \ref{res-08}, we say that \eqref{eq-bdg} admits a \emph{regular} solution if it can be written as
\begin{equation}\label{eq-32}
P= \exp \Big( \eta(X_0,X_1)+\sum _{ s\in\SS} \theta_s(X_s) +\int _{ \TT}p(t,X_t)\, dt \Big) R
\end{equation}
for some   nice enough  functions $ \eta:\XXX\to\RR$, $p:\TX\to\RR$ where $\TT\subset\ii$ is a finite union of  intervals and $ \theta_s$ with $s$ running through a finite subset $\SS= \left\{s_k; 1\le k\le K\right\} \subset (0,1)$.  By ``nice enough'' it is meant that all the   objects and equations built upon $R, \eta$, $p$ and $\theta$ to be encountered  in this section are meaningful or admit solutions and that the functions $\eta,$ $ \theta$ and $p$ are such that for any $x\in\XX,$ the function $ \psi^x:\iX\to\RR$ specified by \eqref{eq-35} below is well-defined, $C^2$ in space and piecewise $C^1$ in time.
\end{definition}

\begin{remarks}\label{rem-01}\ \begin{enumerate}[(a)]

\item
The path measure $P$ defined by \eqref{eq-32} is the solution of a Bredinger problem of the form
\begin{equation}\label{eq-51}
H(Q|R)\to \textrm{min};\qquad Q\in\MO: (Q_t=\mu_t,\  \forall t\in \SS\cup \TT ),\quad Q _{01}=\pi.
\end{equation}

\item
In the important case where the Bredinger problem is \eqref{eq-14}, i.e. the marginal constraint is the incompressibility condition: $P_t=\vol$ for all $t\in\ii,$ we know by Theorem \ref{res-08} that the solution has the form \eqref{eq-28}:
$P=\exp (A(X)+ \eta(X_0,X_1))\, R.$ But
we did not succeed in proving that  the additive functional writes as 
$A(X)=\Iii p(t,X_t)\, dt$ for some \emph{function}  $p$. 

\item
 The expression \eqref{eq-32} corresponds to a measure $ \alpha(dt) =\1 _{\{ t\in\TT\}}\, dt+\sum _{ s\in\SS} \delta_s(dt)$ in \eqref{eq-26}.
\end{enumerate}\end{remarks}

\subsection*{The kinematics of $P$}
While the dynamics of $P$ is specified by formula \eqref{eq-32} which is expressed in terms of potentials $p$ and $ \theta,$ its kinematic description  is specified by the velocity vector field $\vf^P$ appearing in  the stochastic differential equation
\begin{align*}
dX_t=\vf^P_t\,dt+dM^P_t,\qquad P\as
\end{align*}
where $M^P$ is a local $P$-martingale.
We are going to calculate  $\vf^P$ in terms of the potentials $ \theta$ and $p$. This will permit us to establish  some connection between $P$ and the Navier-Stokes equation \eqref{eqac-04} in the special case where $R$ is the Brownian path measure with diffusion constant $a>0$.

\begin{lemma}
The velocity vector field is of the form
\begin{align}\label{eq-33}
\vf^P_t=a \beta^P_t,\quad P\as
\end{align} 
with $ \beta^P$  predictable satisfying $P$-almost surely
\begin{align}\label{eq-34}
\beta^P_t\scal dX_t
	- a/2\   |\beta^P_t|^2\, dt
	=\1 _{ \left\{t\in\SS\right\} } \theta_t+p_t\, dt +d \psi ^{ X_0}_t(X_t),\quad \forall 0\le t\le1,
\end{align}
where  we have set for any $0\le t\le 1$ and $x\in\XX,$
\begin{align}\label{eq-35}
\psi^x(t,z):=\log E_R \Big[ \exp \Big( \eta(x,X_1)+ \sum _{ s\in\SS, s>t} &\theta_s(X_s)\nonumber\\
	&+ \int _{\TT\cap (t,1]}p(r,X_r)\,dr\Big)\,\Big|\, X_t=z \Big],\quad R_t\as 
\end{align}
Note that for $t=1,$ we have $ \psi^x(1,\cdot)=\eta(x,\cdot).$ 
\end{lemma}

The identity \eqref{eq-34} is the keystone of the computation of $\vf^P$. It relates  kinetic terms on the left-hand side with dynamical terms on the right-hand side.

\begin{proof}
Our calculation of  $\vf^P$ is done by confronting the Feynman-Kac type formula \eqref{eq-32} with the expression 
\begin{align}\label{eq-36}
\frac{dP}{dR}= \frac{dP_0}{dR_0}(X_0)
	\exp \Big(\Iii \beta^P_t\cdot dX_t
	- \frac{a}{2} \Iii |\beta^P_t|^2\, dt\Big),\quad P\as
\end{align}
issued from Girsanov's theory, where $ \beta^P$ is a predictable vector field.
To perform this identification, two operations are required.
\begin{enumerate}[(i)]
\item
We disintegrate $R$ and $P$ along their initial positions, i.e.\ $R=\int _{ \Rn}R^x(\cdot)\,R_0(dx)$ and $P=\int _{ \Rn}P^x(\cdot)\,P_0(dx)$, where for any $x\in\Rn,$ we denote $Q^x(\cdot)=Q^x(\cdot\mid X_0=x).$ The main advantage of this disintegration is that it allows to work with the Markov measures  $P^x$, while $P$ is only reciprocal, see \cite{BL16}.

\item
The expressions \eqref{eq-32} and \eqref{eq-36} of $dP/dR$ are not enough.
We shall need for each $0\le t\le 1,$  formulas for the Radon-Nikodym density
\begin{align}\label{eq-37}
\frac{dP^x _{ [0,t]}}{dR^x _{ [0,t]}}=E _{ R^x} \Big( \frac{dP^x}{dR^x}\mid X _{ [0,t]}\Big) 
\end{align} 
of the restrictions  to the $ \sigma$-field $\sigma(X _{ [0,t]}).$
\end{enumerate}

As regards (i), denoting $b^x$  the drift field of the Markov measure $P^x,$ since  $dX_t=\vf^P_t\,dt+M^P_t,\ P\as$  and for $P_0$-almost all $x$ in restriction to $ \left\{X_0=x\right\} $ we have: 
\begin{align*}
dX_t=b^x_t\,dt+dM ^{ P^x}_t=\vf^P_t\,dt+dM^P_t=\vf^P_t\,dt+dM ^{ P^x}_t,\ P^x\as,
\end{align*} 
 we see that $\vf^P$ has the special form
\begin{align*}
\vf^P_t=b ^{ X_0}_t,\ \textrm{ for a.e.\ }t,\quad P\as
\end{align*}
This shows that it is enough to obtain the drift $b^x$ of  the Markov measure $P^x$ for any $x$.

As regards (ii), Girsanov's representation \eqref{eq-36} becomes for all $0\le t\le 1,$
\begin{align*}
\frac{dP^x _{ [0,t]}}{dR^x _{ [0,t]}}
	=\exp \Big(\int_0^t \beta^x_r\cdot dX_r
	- \frac{a}{2} \int_0^t |\beta^x_r|^2\, dr\Big),\quad P^x\as
\end{align*}
with
\begin{align}\label{eq-38}
b^x_t=a \beta^x_t,\quad P^x\as
\end{align} 
for some predictable process $ \beta^x.$
On the other hand,  with \eqref{eq-32}, \eqref{eq-37} and the Markov property of $R,$ we obtain
\begin{align*}
\frac{dP^x _{ [0,t]}}{dR^x _{ [0,t]}}
	=\exp \Big( \sum _{ s\in\SS, s\le t} \theta_s(X_s)
	+\int _{\TT\cap [0,t]} p(r,X_r)\,dr+ \psi^x_t(X_t) \Big) ,\quad R^x\as
\end{align*}
where $ \psi$ is defined at \eqref{eq-35}.
Comparing the differentials of the two expressions of $dP^x _{ [0,t]}/dR^x _{ [0,t]},$ we arrive at
\begin{equation}\label{eq-39}
\beta^x_t\cdot dX_t
	- a/2\  |\beta^x_t|^2\, dt
	=\1 _{ \left\{t\in\SS\right\} } \theta_t
	+ \1 _{ \left\{t\in\TT\right\} }p_t\, dt +d \psi^x_t(X_t),\quad \forall 0\le t< 1,\ P^x\as
\end{equation}
which is \eqref{eq-34}.
\end{proof}

\begin{theorem}\label{res-15}
Let us assume that the dual parameters $\eta,$ $ \theta$ and $p$ are such that $P$ is  regular in the sense of Definition \ref{def-02}. Then,
\begin{align*}
\vf^P_t(X _{ [0,t]})=\vf^P_t(X_0,X_t)=a\nabla \psi ^{ X_0}_t(X_t),\quad \forall 0\le t\le 1,\ P\as,
\end{align*}
where $\nabla\psi ^{ X_0}_t(X_t)$ stands for $\nabla_z \psi^x(t,z) _{ |x=X_0, z=X_t}$ and for any $x\in\XX,$ $ \psi^x$ is given by \eqref{eq-35}.

Moreover, for any $x\in\XX,$ $ \psi^x$  is a classical solution of the second-order Hamilton-Jacobi equation
\begin{equation}\label{eq-40}
\left\{
\begin{array}{ll}
\big[( \partial_t+a \Delta/2)\psi
	+a\ |\nabla \psi|^2/2+\1 _{ \left\{t\in\TT\right\} }p\big](t,z)=0,\qquad	
	 &0\le t<1,\ t\not\in\SS, z\in\XX,\\
 \psi(t,\cdot)- \psi(t^-,\cdot)=- \theta(t,\cdot), &t\in\SS,\\
\psi(1,\cdot)= \eta(x,\cdot), & t=1.
\end{array}
\right.
\end{equation} 
\end{theorem}

\begin{proof}
Let us work $P^x$-almost surely with $X_0=x$ fixed and by means of \eqref{eq-38}, rewrite \eqref{eq-39} as
\begin{align*}
d \psi^x_t(X_t)=
	\beta^x_t\cdot dM ^{ P^x}_t+(a\ | \beta^x_t|^2/2-\1 _{ \left\{t\in\TT\right\} }p_t)\,dt
	-\1 _{ \left\{t\in\SS\right\} } \theta_t,\quad P^x\as
\end{align*}
where Girsanov's theory ensures that
$
dM ^{ P^x}_t=dX_t-b^x_t\,dt
$
is the increment of a local $P^x$-martingale. We see that  $t\mapsto \psi^x_t(X_t)$ is a $P^x$-semimartingale. 
\\
Our regularity assumption ensures that  $ \psi^x$ defined at \eqref{eq-35}  verifies the following Itô formula
\begin{align*}
d \psi_t^x(X_t)&= [\psi^x_t- \psi^x _{ t^-}](X_t)+ \nabla \psi^x_t(X_t)\cdot dX_t+\big( \partial_t+ a \Delta \psi^x_t/2\big)(X_t)\,dt,
	\quad R^x\as \\
	&=[\psi^x_t- \psi^x _{ t^-}](X_t)+ \nabla \psi^x_t(X_t)\cdot dM ^{ P^x}_t+\big( \partial_t+ a \Delta \psi^x_t/2\big)(X_t)\,dt\\
	&\hskip 7.7cm +a|\nabla \psi^x_t|^2(X_t)\,dt,
	 \quad P^x\as 
\end{align*}
The  uniqueness of the decomposition of a semimartingale (Doob-Meyer  theorem) allows for  the identification of the previous two $P^x$-almost sure expressions of $ d \psi^x_t(X_t)$ and gives
\begin{align*}
 &\beta^x_t=\nabla \psi^x_t(X_t)\\
& -\1 _{ \left\{t\in\SS\right\} } \theta_t(X_t)=[\psi^x_t- \psi^x _{ t^-}](X_t)\\
& a | \beta^x_t|^2/2-\1 _{ \left\{t\in\TT\right\} }p_t(X_t)=\big(\partial_t+ a \Delta \psi^x_t/2\big)(X_t)+a|\nabla \psi^x_t|^2(X_t)
\end{align*}
where the first and third equalities hold $dt dP^x$-almost everywhere.  The second one is valid $P^x\as$ and we implicitly identified the  semimartingale $ \psi^x_t(X_t)$ with its càdlàg modification.
As $ \psi^x$ is assumed to be regular  for any $x$, we obtain
\begin{align*}
\beta^P_t( \omega)
	=\nabla \psi ^{ \omega_0}_t( \omega_t),&  &&t\in\ii,\ \omega\in\OO\\
[\psi^x_t- \psi^x _{ t^-}](z)
	=- \theta(t,z),&  &&t\in \mathcal{S},\  x,z\in\XX,\\
( \partial_t+a \Delta/2)\psi^x(t,z)+a |\nabla \psi^x|^2(t,z)/2+\1 _{ \left\{t\in\TT\right\} }p(t,z)
	=0,&  &&t\in[0,1)\setminus\SS,\ x,z\in\XX.
\end{align*}
This completes the proof of the theorem.
\end{proof}

\subsection*{Fluid evolution}

We wish to relate the stochastic velocity field $\vf^P$ with some evolution equation looking like Navier-Stokes equation \eqref{eqac-04}. It is well known in classical mechanics that taking the gradient of Hamilton-Jacobi equation leads to the second Hamilton equation (Newton's equation). Let us do it with the second order Hamilton-Jacobi equation \eqref{eq-40}. For the PDE part, when $ t\not\in\SS$ is not an instant of shock, denoting for any $x\in\XX$ the forward velocity of $P(\cdot\mid X_0=x)$ by  
$$
\vvf{x}{}:=\vf ^{ P(\cdot\mid X_0=x)}=a\nabla \psi^x,
$$ 
we obtain
\begin{align*}
0	=\nabla[ \partial_t \psi^x+a/2\  |\nabla \psi^x|^2+a/2\ \Delta\psi^x+ p]
	&= \partial_t\nabla \psi^x+a\nabla \psi^x\cdot \nabla (\nabla \psi^x)
		+a/2\ \Delta \nabla \psi^x +\nabla p\\
	&= ( \partial_t+\vvf{x}{}\cdot\nabla)\nabla \psi^x+  \Delta( \vvf{x}{})/2+\nabla p
\end{align*}
and multiplying by $a$, we see that
\begin{align*}
\left\{
\begin{array}{ll}
( \partial_t+\vvf{x}{}\cdot\nabla) (\vvf{x}{})= -a \Delta(\vvf{x}{})/2- \nabla p',\qquad & t<1,\ t\not\in\SS,\\
\vvf{x}t-\vvf{x}{ t^-}=-\nabla \theta_t,& t\in\SS,\\
 \vvf{x}1=\nabla _{ y} \eta(x,\cdot) , & t=1,
\end{array}
\right.
\end{align*}
with $p'=ap.$ The left-hand side is the convective acceleration $ \mathrm{D}_t(\vvf{x}{})$ as in Navier-Stokes equation, but besides the gradient $-\nabla p'$ of a pressure in the right-hand side, we have $-a/2\ \Delta\vvf{x}{}$ with the \emph{wrong sign}. 

The forward velocity $\vf^P$ of $P$ does not fulfill our hopes. But we   are going to see that its backward velocity $\vb^P$ does. 
Recall \eqref{eq-11} and \eqref{eq-12} for the definitions of these velocities. 
Let us introduce
$$
\vvb{y}{}:=\vb ^{ P(\cdot\mid X_1=y)}
$$
and for any $0\le \alpha\le 1,$
\begin{align}\label{eq-41}
v ^{ \alpha}_t(z)
	:= E_P \big[(1- \alpha)\vvf{X_0}t + \alpha\vvb{X_1}t\mid X_t=z\big]
	=(1- \alpha)\vf_t(z)+ \alpha\vb_t(z)
\end{align}
with
\begin{align*}
\vf_t(z):=E_P[ \vvf{X_0}t\mid X_t=z]=\IRn \vvf xt\, P(X_0\in dx\mid X_t=z)\\
\vb_t(z):=E_P[ \vvb{X_1}t\mid X_t=z]=\IRn \vvb yt\, P(X_1\in dy\mid X_t=z)
\end{align*}
the average forward and backward velocities. In particular,  $ \alpha =1/2$ corresponds to the current velocity
\begin{align*}
v^{\mathrm{cu}}_t:=\big(\vf_t+\vb_t\big)/2=v ^{ \alpha=1/2}.
\end{align*}
The reason for calling $v ^{ \alpha=1/2}$ the current velocity is that, among all the $v ^{ \alpha}$'s,  it is the only  one satisfying the continuity equation \eqref{eq-43} below, see \eqref{eq-50}.

\begin{theorem}
For any $y\in\XX$, the backward velocity field  $\vvb{y}{}$  of $P(\cdot\mid X_1=y)$ solves 
\begin{align}\label{eq-42}
\left\{
\begin{array}{ll}
( \partial_t+\vvb{y}{}\cdot\nabla) \vvb{y}{}= a \Delta\vvb{y}{}/2- \nabla p',\qquad & t>0,\ t\not\in\SS,\\
\vvb{y}t-\vvb{y}{ t^-}=\nabla \theta_t,& t\in\SS,\\
 \vvb{y}0=-\nabla _{ x} \eta(\cdot,y) , & t=0,
\end{array}
\right.
\end{align}
with $p'=ap$ and $\eta^y=\eta(\cdot,y).$ 
On the other hand, the current velocity $v^{\mathrm{cu}}$ satisfies the continuity equation
\begin{align}\label{eq-43}
\partial_t  \mu+\nabla\scal ( \mu v^{\mathrm{cu}})=0.
\end{align}
Moreover, 
\begin{align*}
\vvb y{}_t(z)=a\nabla \varphi^y_t(z),
\quad t\not\in \mathcal{S},
\end{align*}
where $ \varphi^y$ solves the Hamilton-Jacobi-Bellman equation
\begin{equation}\label{eq-44}
\left\{
\begin{array}{ll}
( \partial_t-a \Delta/2) \varphi
	+a |\nabla  \varphi|^2/2+p=0,\qquad
	 &t>0,\ t\not\in\SS,\\
 \varphi(t,\cdot)- \varphi(t^-,\cdot)=\theta(t,\cdot), &t\in\SS,\\
\varphi(0,\cdot)= -\eta(\cdot,y), & t=0.
\end{array}
\right.
\end{equation} 
\end{theorem}

The first equation of the system \eqref{eq-42} is the desired Newton part of the Navier-Stokes equation 
(Burgers equation) with the right positive sign in front of the viscous force term: $a \Delta\vvb{y}{}/2,$
see \eqref{eqac-04}. The continuity equation \eqref{eq-43} is the analogue of $\nabla\scal v=0$ in \eqref{eqac-04} 
which corresponds to the case $\mu\equiv1$.

\begin{proof}
Introducing the time-reversed 
\begin{align*}
P^*:={X^*}\pf P
\end{align*}
of $P$,  where $X^*_t:=X _{ 1-t},$ $0\le t\le 1,$ we obtain $\vb^P_t(X _{ [t,1]})=-[\vf ^{ P^*}_{ 1-t}\circ X^*](X_{[0,1-t]}).$  On the other hand, as $P$ is reciprocal, so is $P^*$.  Consequently, the forward and backward velocities $\vf^P_t(X _{ [0,t]})=\vf^P(X_0,X_t)$ and $\vb^P_t(X _{ [t,1]})=-[\vf ^{ P^*}_{ 1-t}\circ X^*](X_{ [t,1]})=-[\vf ^{ P^*}_{ 1-t}\circ X^*](X _{t},X_1)=\vb^P_t(X_t,X_1)$  only depend on the  states $X_0,X_t,X_1$ and can  be considered simultaneously in a sum or a difference without assuming the knowledge of the whole sample path. Let us emphasize for future use the identity
\begin{align*}
\vb^P_t(X_t,X_1)=-[\vf ^{ P^*}_{ 1-t}\circ X^*](X_t,X_1),\quad 0\le t\le 1.
\end{align*}
Since $R$ is assumed to be reversible, i.e.\ $R=R^*,$ we see that $dP^*/dR=(dP/dR^*)\circ X^*=(dP/dR)\circ X^*$ and we obtain with \eqref{eq-32} that
\begin{align*}
P^*	=\exp\Big(\eta^*(X_0,X_1)+ \sum _{ s\in \SS^*} \theta^* _{ s}(X _{ s})
	+\Iii p^*(t,X _{ t})\,dt\Big)\, R
\end{align*}
with $\eta^*(x,y)=\eta(y,x)$ for all $x,y\in\XX,$  $ \theta^*_s= \theta _{ 1-s}$ for all $s\in\SS^*= \left\{1-s; s\in\SS\right\} $ and $p^*(t,\cdot)=p(1-t,\cdot)$ for all $0\le t\le 1.$ Applying Theorem \ref{res-15} to $P^*$, we see that 
\begin{align*}
\vf ^{ P^*}_t(X _{ [0,t]})=a\nabla \xi ^{ X_0}_t(X_t)
\end{align*}
with $\xi^y$ solution of
\begin{equation*}
\left\{
\begin{array}{ll}
( \partial_t+a \Delta/2) \xi
	+a |\nabla  \xi|^2/2+p^*=0,\qquad
	 &t<1,\ t\not\in\SS^*,\\
 \xi(t,\cdot)- \xi(t^-,\cdot)=- \theta^*(t,\cdot), &t\in\SS^*,\\
\xi(1,\cdot)= \eta^*(y,\cdot), & t=1.
\end{array}
\right.
\end{equation*} 
Therefore, setting 
$	
\varphi^y(t,\cdot)=- \xi^y(1-t,\cdot),
$	
we obtain
\begin{align}\label{eq-45}
\vvb{X_1}t(X_t)=a\nabla \varphi_t^{ X_1}(X_t)
\end{align}
with $ \varphi^y$ solution of \eqref{eq-44}.
Taking the gradient of this equation and multiplying by $a$, we see that for any $y\in\XX$, $\vvb{y}{}$  solves \eqref{eq-42}.

Of course the  marginal constraint cannot be verified by the velocities $\vvf{x}{}$ and $\vvb{y}{}$, since they start or arrive at a Dirac mass. One must consider averages of these fields as in \eqref{eq-41} to recover this constraint. 
For any smooth bounded functions, we have
\begin{align*}
\IRn u&\,d(P_t-P_0)
	=E_P\int_0^t \left[ \vvf{X_0}s(X _{s})\cdot \nabla u(X_s)+\ud \Delta u(X_s)\right] \,ds\\
&= E_P\int_0^t \left[ \vf_s(X _{s})\cdot \nabla u(X_s)+\ud \Delta u(X_s)\right] \,ds
	= \int_0^tds\IX [\vf_s\cdot \nabla u+\ud \Delta u](z)\,P_s(dz)
\end{align*}
and
\begin{align*}
\IRn u&\,d(P_1-P_t)
	=E_P\int_t^1 \left[ \vvb{X_1}s(X _{s})\cdot \nabla u(X_s)-\ud \Delta u(X_s)\right] \,ds\\
&= E_P\int_t^1 \left[ \vb_s(X _{s})\cdot \nabla u(X_s)-\ud \Delta u(X_s)\right] \,ds
= \int_t^1ds\IX [\vb_s\cdot \nabla u-\ud \Delta u](z)\,P_s(dz)
\end{align*}
implying 
$	
\left\langle u, \partial_t \mu \right\rangle 
	= \left\langle \vf_t\cdot \nabla u+ \Delta u/2, \mu_t \right\rangle
	= \left\langle \vb_t\cdot \nabla u- \Delta u/2, \mu_t \right\rangle.
$	
It follows that for any $0\le \alpha\le 1,$
$	
\left\langle u, \partial_t \mu_t \right\rangle 
 	= \left\langle v ^{ \alpha}_t\scal\nabla u+ a (1/2- \alpha)\Delta u , \mu_t\right\rangle
$	
which is equivalent to
\begin{align}\label{eq-50}
 \partial_t \mu + \nabla\cdot ( \mu v ^{ \alpha})
 	= a (1/2- \alpha)\Delta \mu.
\end{align}
In particular, taking $ \alpha=1/2$ leads to \eqref{eq-43} and completes the proof of the proposition.
\end{proof}

\begin{remark}[\emph{The pressure does not depend on the final position $y$}]

It is an important consequence of Theorem \ref{res-08} that the   pressure $p$ and the potential $ \theta$ do not depend on the final position $X_1.$ The only explicit appearance of $X_1$ is in the function $ \eta.$ Consequently, the pressure $p'$ in the Burgers equation \eqref{eq-42} only depends on the actual position. This means that all the fluid particles are submitted to the same pressure field $\nabla p'$ regardless of their final positions $y$. A similar remark is valid for the shock potential $ \theta.$
\end{remark}

\begin{remark}[\emph{A mixture of flows tagged by their final positions}]
The solution $P$ of Bredinger's problem is well described as the statistical mixture 
\begin{align}\label{eq-46}
P(\cdot)=\IX  \Py(\cdot)\, \mu_1(dy)
\end{align}
 where $\Py:=P(\cdot\mid X_1=y)$ admits the gradient drift field $\vvb y{}=\nabla \varphi^y$. This velocity field is completely specified by \eqref{eq-44} where the endpoint target  $y$ only occurs in the initial condition via the function $- \eta(\cdot,y).$ Formula \eqref{eq-46} is a superposition principle. Each particle ending at $y$ is subject to the gradient backward velocity field $\vvb y{}$ solving the Burgers equation \eqref{eq-42} and the volume constraint $P_t= \mu_t,$  $\forall t\in \SS\cup\TT,$ (recall \eqref{eq-51})
is recovered superposing  all the  flows tagged by their final positions, via formula \eqref{eq-46}.  This superposition phenomenon is very reminiscent of the structure of the multiphase vortex sheets model encountered in \cite{Bre97}.
\end{remark}

\begin{remark}[\emph{The average velocity is not a gradient}]

The incompressibility constraint applied to a gradient velocity field $v=\nabla \theta$ on the torus $\Tn$ reads as $0=\nabla v=\nabla\cdot \nabla \theta= \Delta \theta.$ But this implies that   $v$ vanishes everywhere. This is the reason why knowing that the average velocity is not a gradient leaves some room in our model.
\\
We know with \eqref{eq-45} that $\vvb yt(z)=\nabla_z \varphi^y(z)$ is a gradient field.  Consequently, the average backward velocity writes as
\begin{equation*}
\vb_t(z) = \int \nabla_z \varphi_t^y(z)\,P_1 ^{ tz}(dy)
\end{equation*}
and we see that the dependence on $z$ of $P_1 ^{ tz}:=P(X_1\in dy\mid X_t=z)$ prevents us from identifying $\vb_t(z) $ with
$  \nabla_z[\int \varphi^y_t(z)\,P_1 ^{ tz}(dy)].$
Introducing the average potential
\begin{equation*}
\varphi^P_t(z):=\int \varphi_t ^{ y}(z)\, P_1 ^{ tz}(dy),
\end{equation*}
we obtain 
\begin{equation*}
\vb_t(z)=\nabla  \varphi_t^P(z)- \int  \varphi_t^y(z)\,\nabla_z P_1 ^{ tz}(dy).
\end{equation*}
\end{remark}

\section{Existence of a solution  on $\Tn$}\label{sec-existence}

We are going to prove  a sufficient condition of existence of a solution of the Bredinger problem \eqref{eq-14} in the special important  case where  the reference path measure $R\in\PO$ is the reversible Brownian motion on the  the flat torus $ \XX=\Tn$ and $ \mu_t=\vol,$ for all $t$. We refer to this problem as
\begin{equation}\label{eq-47}
H(P|R)\to \textrm{min}; [P_t=\vol, \forall 0\le t\le 1],\ P _{ 01}= \pi.
\tag{H$ _{ \Tn}$}
\end{equation}  
It is an adaptation of a result in \cite{Bre89} of existence of a generalized incompressible flow in $\Tn$. The specific property of the reversible Brownian motion $R$ is the translation invariance 
\begin{equation}\label{eq-48}
R =R (x+\cdot ),\quad \forall x\in\Tn.
\end{equation}
Combined with the translation invariance of $\vol$ (which is implied by \eqref{eq-48}), this will lead us to the desired result.
All we have to  find is some path measure $Q\in\PO$ which satisfies the constraints  $[Q_t=\vol, \forall 0\le t\le 1],$ $Q _{ 01}= \pi$ and such that $H(Q|R)< \infty.$ The path measure of interest is 
$$
Q=\int _{ \XX^3} R(\cdot\mid X_0=x,X _{ 1/2}=z,X_1=y)\, \gamma(dxdzdy)
$$
with 
$ \gamma(dx dz dy)=\pi(dxdy)\vol (dz)$ in $ \mathrm{P}(\XX^3).$

\begin{proposition}\label{res-05}
The path measure $Q$ satisfies the constraints  $[Q_t=\vol, \forall 0\le t\le 1]$ and $Q _{ 01}= \pi$. If $H( \pi| R _{ 01})< \infty,$ then $H(Q|R)< \infty.$ 
\end{proposition}

\begin{corollary}\label{res-07}
The entropy minimization problem \eqref{eq-47} admits a unique solution if and only if $H( \pi|R _{ 01})< \infty$. 
\end{corollary}

\begin{proof}[Proof of Corollary \ref{res-07}]
If $H( \pi| R _{ 01})< \infty,$ by Proposition \ref{res-05}  we have $\inf \eqref{eq-bdg}\le H(Q|R)< \infty$ and we conclude with Corollary \ref{res-03}-(1) that \eqref{eq-47} admits a unique solution.
Conversely, when \eqref{eq-47} admits a solution $P$, we have $H( \pi|R _{ 01})=H((X_0,X_1)\pf P|(X_0,X_1)\pf R)\le H(P|R)< \infty.$
\end{proof}

\begin{proof}[Proof of Proposition \ref{res-05}]

As $R$ is Markov, it satisfies
\begin{alignat*}{2}
R(\cdot\mid X_0=x,X _{ 1/2}&=z,X_1=y)\\=\  &R(X _{ [0,1/2]}\in\cdot\mid X_0=x,X _{ 1/2}=z)R(X _{ [1/2,1]}\in\cdot\mid X _{ 1/2}=z,X_1=y).
\end{alignat*}
Let us check that $Q$ satisfies the announced constraints.
\\
We have $Q _{ 01}= \pi$ since for any measurable subsets $A$ and $B$ of $\XX,$ 
\begin{eqnarray*}
Q _{ 01}(A\times B)&=&Q(X_0\in A, X_1\in B)\\
	&=& \int _{ \XX^3} R(X _{0}\in A\mid X_0=x,X _{ 1/2}=z)R(X _{ 1}\in B\mid X _{ 1/2}=z,X_1=y)\, \gamma(dxdzdy)\\
	&=& \int _{ \XX^3} \1 _{ x\in A}\1 _{ y\in B}\, \gamma(dxdzdy)= \gamma(A\times\XX\times B)=\vol(\XX) \pi(A\times B)\\
	&=&\pi (A\times B).
\end{eqnarray*}
Let us show that for all $0\le t\le 1,$ $Q_t=\vol.$ Take $0\le t\le 1/2$ and denote $R(X _{ [0,1/2]}\in\cdot\mid X_0=x, X _{ 1/2}=z)=\widetilde{R} ^{ x,z}(\cdot).$
 Since $ \pi(\cdot\times \XX)= \pi(\XX\times\cdot)=\vol,$ we have $ \gamma(dxdz\times \XX)=\pi(dx\times\XX)\vol(dz)=\vol(dx)\vol(dz).$ Hence, for any measurable bounded function $f$ on $\XX$, we have 
 \begin{eqnarray*}
\IX f\, dQ_t&=&\int _{ \XX^3}E _{ \widetilde{R} ^{ x,z}}[f(X_t)]\, \gamma(dxdzdy)
 	= \int _{ \XX^2} E _{\widetilde{R} ^{ x,z}}[f(X_t)]\, \vol(dx)\vol(dz)\\
	&\overset {\eqref{eq-48}}{=}& \int _{ \XX^2} E _{\widetilde{R} ^{ 0,z-x}}[f(X_t-x)]\, \vol(dx)\vol(dz)
	= \int _{ \XX^2} E _{\widetilde{R} ^{ 0,a}}[f(X_t-x)]\, \vol(dx)\vol(da)\\
	&=&  \int _{ \XX} E _{\widetilde{R} ^{ 0,a}}\Big[\IX f(X_t-x)\, \vol(dx)\Big]\,\vol(da)
	=  \int _{ \XX} E _{\widetilde{R} ^{ 0,a}}\Big[\IX f\, d\vol\Big]\,\vol(da)\\
	&=&\IX f\, d\vol,
 \end{eqnarray*}
 where the translation invariance of $\vol$  was used at the last but one equality.
This shows that $Q_t=\vol$ for all $0\le t\le 1/2$. A similar argument works for $1/2\le t\le 1.$

It remains to compute the entropy $H(Q|R)$ to obtain a criterion of existence of a solution. Let us denote $Q _{ 0,1/2,1}(dxdzdy):=Q(X_0\in dx,X _{ 1/2}\in dz, X_1\in dy)$ and $Q ^{ xzy}:=Q(\cdot\mid X_0=x,X _{ 1/2}=z,X_1=y).$   
We have
\begin{eqnarray*}
H(Q|R)&\overset{(i)}=&H(Q _{ 0,1/2,1}|R _{ 0,1/2,1})+\int _{ \XX^3}
H(Q ^{ xzy}|R ^{ xzy})Q _{ 0,1/2,1}(dxdzdy)\\
&\overset{(ii)}=&H( \gamma|R _{ 0,1/2,1})\\
&\overset{(iii)}=& H( \gamma _{ 01}|R _{ 01})+\IXX H(\gamma ^{ xy}|R ^{ xy} _{ 1/2})\, \pi(dxdy)\\
 &\overset{(iv)}=& H( \pi|R _{ 01})+\IXX H(\vol|R ^{ xy} _{ 1/2})\, \pi(dxdy).
\end{eqnarray*}
The factorization property of the entropy  is invoked at the equalities (i) and (iii). The identity  (ii) is a consequence of
$Q ^{ xzy}=R ^{ xzy},$ for $ \gamma=Q _{ 0,1/2,1}$-almost 
all $(x,z,y).$
The last equality (iv) follows from $ \gamma _{ 01}(dxdy):= \gamma(X\in dx, Y\in  dy)= \pi(dxdy)\vol(\XX)= \pi(dxdy)$ and $ \gamma ^{ xy}(dz):= \gamma(Z\in dz\mid X=x,Y=y)= \gamma(Z\in dz)=\vol(dz)$ since $(X,Y)$ and $Z$ are $ \gamma$-independent. It remains to show that
\begin{equation}\label{eq-49}
\sup _{ x,y\in\Tn}H(\vol|R ^{ xy}_{ 1/2})< \infty,
\end{equation}
to obtain that  $H(Q|R)$ is finite as soon as $H( \pi| R _{ 01})< \infty.$
By means of the formula
\begin{equation*}
 \frac{dR ^{ xy}_{ 1/2}}{d\vol}(z)= (2 /\pi) ^{ n/2} 
 	\frac{\sum _{ k,l\in \mathbb{Z}^n}\exp(- |z-x+k|^2-|y-z+l|^2 )}{\sum _{ k\in \mathbb{Z}^n}\exp(-|y-x+k|^2/2)},
\end{equation*}
we see that the function
\begin{alignat*}{1}
H(\vol|R ^{ xy}_{ 1/2})=\log \Big[ (\pi/2) ^{ n/2}&\sum _{ k\in \mathbb{Z}^n}\exp(-|y-x+k|^2/2)\Big] \\
&-\int _{ \Tn}\log \left(\sum _{ k,l\in \mathbb{Z}^n}\exp(- |z-x+k|^2-|y-z+l|^2 )\right) \,\vol(dz)
\end{alignat*}
is continuous in $(x,y).$ As $\Tn$ is compact, we have proved \eqref{eq-49}. This completes the proof of the proposition.
\end{proof}

\vskip 5mm
\bf Acknowledgements \rm

The second and the fourth authors were partially supported by the FCT Portuguese project PTDC/MAT-STA/0975/2014.


\end{document}